\documentclass{amsart}
\usepackage{xypic}
\usepackage{amssymb}
\usepackage{pstricks}
\usepackage{pst-plot}
\usepackage{nicefrac}
\usepackage{pstricks-add}
\usepackage{booktabs}
\usepackage{multirow}
\usepackage{bigdelim}

\psset{unit=0.35cm}
\SpecialCoor
\psset{gridlabels=0,gridwidth=1.3pt,subgridwidth=0.6pt,subgriddiv=1,griddots=1,subgriddots=1,subgridcolor=white}

\usepackage{amsmath}
\usepackage{amsfonts}
\usepackage{amsthm}
\usepackage{hyperref}
\usepackage{amscd}
\usepackage{color}
\usepackage{booktabs}

\newcommand{\smax}{{\s_{\mathrm{max}}}}
\newcommand{\dual}{\check{\ }}
\newcommand{\FF}{\mathbb{F}}
\newcommand{\ZZ}{\mathbb{Z}}
\newcommand{\GG}{\mathbb{G}}

\newcommand{\PP}{\mathbb{P}}

\newcommand{\redu}{\mathrm{red}}
\newcommand{\A}{\mathbb{A}}
\newcommand{\mcL}{\mathcal{L}}

\newcommand{\QQ}{\mathbb{Q}}
\newcommand{\NN}{\mathbb{N}}
\newcommand{\CO}{\mathcal{O}}
\newcommand{\CI}{\mathcal{I}}
\newcommand{\SHom}{\mathcal{H}\kern -.5pt om}
\newcommand{\D}{\mathcal{D}}
\newcommand{\E}{\mathcal{E}}

\newcommand{\I}{\mathcal{I}}
\newcommand{\Ysep}{{Y^{\mathrm{sep}}}}
\newcommand{\Deltasep}{{\Delta^{\mathrm{sep}}}}

\newcommand{\CH}{\mathfrak{X}}
\newcommand{\s}{\mathfrak{s}}
\newcommand{\mcA}{\mathcal{A}}
\newcommand{\R}{\mathcal{R}}

\let\div\relax
\DeclareMathOperator{\div}{div}
\DeclareMathOperator{\Div}{Div}
\DeclareMathOperator{\CaDiv}{CaDiv}
\DeclareMathOperator{\Hom}{Hom}

\DeclareMathOperator{\Cl}{Cl}
\DeclareMathOperator{\spec}{Spec}

\DeclareMathOperator{\supp}{supp}
\DeclareMathOperator{\Bl}{Bl}

\DeclareMathOperator{\chara}{char}
\DeclareMathOperator{\codim}{codim}

\usepackage{enumerate}
\newtheorem*{mainthm}{Main Theorem}
\newtheorem{thm}{Theorem}[section]
\newtheorem{prop}[thm]{Proposition}
\newtheorem{lemma}[thm]{Lemma}
\newtheorem{cor}[thm]{Corollary}

\newtheorem{lem}[thm]{Lemma}
\theoremstyle{definition}

\newtheorem{defn}[thm]{Definition}
\newtheorem{rem}[thm]{Remark}
\newtheorem{question}[thm]{Question}
\newtheorem{ex}[thm]{Example}

\title[$F$-split and $F$-regular Varieties]{$F$-split and $F$-regular Varieties with a Diagonalizable Group Action}

\author{Piotr Achinger}
\address{Department of Mathematics, University of California,
Berkeley, CA 94720, USA}
\email{achinger@math.berkeley.edu}

\author{Nathan Ilten}
\address{Department of Mathematics, Simon Fraser University, 8888 University Drive, Burnaby BC V5A 1S6, Canada}
\email{nilten@sfu.ca}

\author{Hendrik S\"u\ss{}}
\address{School of Mathematics, University of Manchester, Oxford Road, Manchester M13 9PL, UK}
\email{suess@sdf-eu.org}
\begin{document}

\begin{abstract}
Let $H$ be a diagonalizable group over an algebraically closed field $k$ of positive characteristic, and $X$ a normal $k$-variety with an $H$-action. Under a mild hypothesis, e.g. $H$ a torus or $X$ quasiprojective, we construct a certain quotient log pair $(Y,\Delta)$ and show that $X$ is $F$-split ($F$-regular) if and only if the pair $(Y,\Delta)$ if $F$-split ($F$-regular).
We relate splittings of $X$ compatible with $H$-invariant subvarieties to compatible splittings of $(Y,\Delta)$, as well as discussing diagonal splittings of $X$. We apply this machinery 
to analyze the $F$-splitting and $F$-regularity of complexity-one $T$-varieties and toric vector bundles,  among other examples.
\end{abstract}

\maketitle

\section{Introduction}
Let $k$ be an algebraically closed field of positive characteristic $p$. An \emph{$F$-splitting} of a $k$-scheme $X$ is an $\CO_X$-linear map splitting the map $F^*:\CO_X\to F_*\CO_X$ induced by the absolute Frobenius morphism; $X$ is \emph{$F$-split} if such a splitting exists.
Originally introduced by Mehta and Ramanathan in their study of Schubert varieties \cite{mehta:85a}, a scheme being $F$-split has remarkable consequences, including the vanishing of all higher cohomology groups of any ample line bundle.
The slightly stronger notion of (global) $F$-regularity\footnote{What we call $F$-regularity is called global $F$-regularity in \cite{smith:00a}.} \cite{hochster:88,smith:00a}
(see Definition \ref{defn:frob})  is closely connected to the property of being log-Fano \cite{schwede:10a}. Both notions have been extended to pairs $(X,\Delta)$ of a normal variety $X$ and an effective $\QQ$-divisor $\Delta$  \cite{schwede:10a}.

In this article, we study the $F$-splitting and $F$-regularity properties of normal varieties equipped with an effective action by a diagonalizable group. On one end of the spectrum, normal toric varieties are always $F$-regular \cite{smith:00a}. On the other hand, characterizations of $F$-split and $F$-regular normal singularities with good $\GG_m$ action have been given by Watanabe \cite{watanabe:91a} in terms of their Demazure representations. Moving to the case $H$ finite, any elliptic curve $E$ can be realized as a double cover of $\PP^1$ (as long as $\chara k \neq 2$)  inducing a $\mu_2$-action, and there is a classical characterization in terms of this cover when $E$ if $F$-split (see Example \ref{ex:elliptic}). Our main result, which we state below, allows us to uniformly treat these three above cases, along with those of many other varieties, including toric vector bundles.

Let $X$ be a normal variety with an effective action by a diagonalizable group $H$. 
Let $X^\circ$ be the open subvariety of $X$ consisting of those points with finite stabilizers, and assume that $X^\circ$ admits a geometric quotient $\pi:X^\circ\to Y$, $Y=X^\circ/H$. This is the case if e.g.~$H$ is a torus, or $X$ is quasiprojective. We define an effective $\QQ$-divisor $\Delta$ on $Y$
by
$$
\sum_{P\subset Y}\frac{\mu(P)-1}{\mu(P)} P,
$$
where $\mu(P)$ is the order of the stabilizer of the generic point of any irreducible component of $\pi^{-1}(P)\subset X^\circ$. 

\begin{mainthm}[Theorem \ref{thm:main}]
Let $X$ be an $H$-variety as above. Then $X$ is $F$-split ($F$-regular) if and only if the pair $(Y,\Delta)$ is $F$-split ($F$-regular). 
\end{mainthm}
The machinery we develop actually gives a bijection between $H$-invariant $F$-splittings of $X$ and $F$-splittings of $(Y,\Delta)$, as well as giving a partial description of the set of all $F$-splittings of $X$ in terms of the quotient pair $(Y,\Delta)$ (see Remark \ref{rem:bijection}). Furthermore, we relate $F$-splittings of $X$ compatible with $H$-invariant subvarieties to certain splittings of $(Y,\Delta)$ (Propositions \ref{prop:comp1} and \ref{prop:comp2}).
The main obstruction to applying our main theorem in practice is that the quotient $Y$ is potentially non-separated. To deal with this, we show that $(Y,\Delta)$ can be replaced by a pair $(\Ysep,\Deltasep)$ such that $\Ysep$ is a variety, and $(Y,\Delta)$ is $F$-split ($F$-regular) if and only if $(\Ysep,\Deltasep)$ is, see Proposition \ref{prop:sep} and Theorem \ref{thm:sep}. 

The main theorem has a number of applications. We recover that normal toric varieties are $F$-regular, along with Watanabe's characterization of  normal singularities with $\GG_m$ action which are $F$-split or $F$-regular (Theorem \ref{thm:watanabe}). 
Given a torus $T$, a \emph{complexity-one} $T$-variety is a $T$-variety $X$ for which $\dim T=\dim X-1$; we give an explicit characterization of $F$-split and $F$-regular complexity-one $T$-varieties, see Theorem \ref{thm:compone}. We also are able to give combinatorial criteria for the $F$-splitting or $F$-regularity of a large class of toric vector bundles. In particular, we characterize $F$-split and $F$-regular rank two vector bundles (Corollary \ref{cor:rktwo}), recover Xin's result \cite{xin:14a} that the cotangent bundle of a smooth toric variety is $F$-split (Corollary \ref{cor:cotangent}), and answer a question of Lauritzen by providing an example of an $F$-split toric vector bundle $\E$ such that $\E^*$ is not $F$-split (Example \ref{ex:lauritzen}).
Further applications include a better understanding of the $F$-splitting and $F$-regularity of cyclic covers (\S\ref{sec:cyclic}), $H$-varieties with toroidal affine quotients (\S\ref{sec:affine}), surjectively graded algebras (\S\ref{sec:surj}), and Cox rings (\S\ref{sec:cox}).

We also study \emph{diagonal splittings} of a $T$-variety $X$, that is, splittings of $X\times X$ which are compatible with the diagonal. Payne showed that normal toric varieties are not always diagonally split, and gave a combinatorial characterization of those which are \cite{payne:09a}. We give a necessary and sufficient criterion for a $T$-invariant splitting of $X\times X$  to be compatible with the diagonal, generalizing Payne's result to higher complexity $T$-varieties, see Theorem \ref{thm:diag-criterion}.
While certainly less explicit than Payne's characterization of diagonally split toric varieties, our criterion can be effectively applied in many instances, particularly for complexity-one $T$-varieties. We also deduce two easier-to-check necessary criteria for the existence of a diagonal splitting.

The rest of the paper is organized as follows. In \S \ref{sec:primer-torus-actions},
we discuss the action of a diagonalizable group on a normal variety, as well as  constructing the log pair $(Y, \Delta)$. Preliminaries on the Frobenius morphism are contained in \S \ref{sec:frob}. We prove our main result in \S \ref{sec:tfrob}, and discuss invariant compatible splittings in \S \ref{sec:compat}.  We show how to replace our potentially non-separated quotient $Y$ by a variety in \S \ref{sec:sep}. In \S \ref{sec:special}, we consider a number of special cases: cyclic covers, $T$-varieties with toroidal affine quotients, $\GG_m$ actions, complexity-one actions, surjectively graded algebras, and Cox rings. We dedicate all of \S \ref{sec:vb} to the special case of toric vector bundles. Finally, \S \ref{sec:diag} contains our results on diagonal splittings of $T$-varieties.

\subsection*{Acknowledgements}
The authors would like to thank Kevin Tucker and Karl Schwede for helpful conversations. The first author's work was supported by Polish National Science Centre (NCN) contract number 2012/07/B/ST1/03343. The second and third authors would like to thank ICMS for research in groups support. 

\section{Diagonalizable Group Scheme Actions}
\label{sec:primer-torus-actions}
\subsection{Preliminaries}
We will work over an algebraically closed field $k$. 
Let $H$ be a \emph{diagonalizable group scheme} over $k$, that is, a subgroup scheme of a torus $\mathbb{G}_m^r$ for some $r\geq 0$. For general facts about diagonalizable group schemes, see e.g.~\cite[Exp. I \S 4.4]{sga3}, \cite[\S 2.2]{waterhouse:79}, \cite{jantzen:03}.  Thus $H$ is isomorphic to a product of copies of the multiplicative group $\mathbb{G}_m$ and group schemes of $n$-th roots of unity $\mu_n = \mathbb{G}_m[n]$. We denote by $M$ the character group $\CH(H)=\Hom_{k-\text{gp.sch.}}(H, \mathbb{G}_m)$ of $H$.
By an \emph{$H$-variety} we mean a normal variety\footnote{An integral separated scheme of finite type over $k$.} $X$ together with an \emph{effective} action $H\times X\to X$. We say that $H$ acts \emph{almost freely} if for all $x\in X(k)$, the stabilizer $H_x$ is finite (as a group scheme over $k$). Note that the set of all points $x\in X$ such that $H_x$ is finite forms an open subvariety $X^\circ$ of $X$ which we call the \emph{almost-free locus}. We will always suppose that the following holds:
\begin{equation} \label{eqn:cover-condition}
 X^\circ\textit{ admits an open cover by }H\textit{-invariant affine open subsets.}
 \end{equation} 
Condition \eqref{eqn:cover-condition} is not automatically fulfilled (see e.g.~\cite[B.3.4.1]{hartshorne:77a}), but it is always satisfied if $H$ is connected \cite{sumihiro:74a}, or if $X$ is quasi-projective: 

\begin{lemma}
Condition \eqref{eqn:cover-condition} is fulfilled if $X$, or more generally, $X^\circ$ is quasi-projective.
\end{lemma}
\begin{proof}
The group scheme $H$ splits splits as a product $H\cong H_\redu^0\times G$, where $H_\redu^0$ is the reduced connected component of the identity, and $G$ is finite. If $X^\circ$ is quasi-projective, it is well known \cite[Theorem 4.3.1]{birula:02a} that there is a good geometric quotient $X^\circ\to X^\circ/G$, where  $Y=X^\circ/G$ is quasi-projective. Furthermore, $Y$ is normal \cite[pp.~126]{shafarevich:13a}. Since the action of $H_\redu^0$ on $X^\circ$ commutes with that of $G$, it descends to an action on $Y$, and $Y$ has an $H_\redu^0$-invariant affine cover by \cite{sumihiro:74a}. Pulling this back to $X^\circ$ gives the necessary $H$-invariant affine cover.
\end{proof}

Suppose now that the $H$-action on $X$ is almost free, that is $X=X^\circ$. In this situation, there is a normal (potentially non-separated) scheme $Y=X/H$ which is a geometric quotient of $X$. We denote by $\pi:X\to Y$ the quotient map. Let $\mathcal{A}=\pi_* \CO_X$, with the associated $M$-grading $\mathcal{A} = \bigoplus_{u\in M} \mathcal{A}_u$, so that $X=\spec_Y \mathcal{A}$. Our first goal is to describe the $H$-variety $X$, or equivalently the graded algebra $\mathcal{A}$, in terms of divisors on $Y$. We treat the case of tori first. Let $\CaDiv_\QQ Y$ denote the group of $\QQ$-divisors on $Y$ with Cartier multiple.

\begin{prop}\label{prop:p-div}
Let $T$ be a torus, and let $X=\spec_Y \mathcal{A}$, be a $T$-variety with an almost free action with quotient $\pi:X\to Y$. Then there exists a homomorphism $\D:\CH(T)\to \CaDiv_\QQ Y$ and a $T$-equivariant isomorphism 
\[ X \cong \spec_Y \bigoplus_{v\in \CH(T)} \CO_Y(\lfloor \D(v) \rfloor)\cdot\chi^v. \]
\end{prop}

The above proposition follows almost immediately from Theorem 3.4  of \cite{altmann:06a}. However, the authors of loc.~cit.~only state and prove this theorem for the case that the ground field has characteristic zero. We believe that their proof applies essentially unchanged in the case of positive characteristic. Instead of verifying all the details here, we present a slightly different argument here for the special case in which we are interested.

\begin{lemma}\label{lemma:free}
Suppose that a torus $T$ acts freely on a $T$-variety $V$. Then $V$ is a Zariski locally trivial $T$-torsor over $V/T$.
\end{lemma}
\begin{proof}
By Luna's \'etale slice theorem \cite{luna:73}, which holds for tori in arbitrary characteristic (cf. \cite[Remark 1.1]{alper:10}), there is an \'etale cover $Y\to V/T$ such that the pullback of $V$ to $Y$ is a $T$-torsor. But by \'etale descent for tori (cf.~e.g.~\cite[III.4]{milne:80a}), $V\to V/T$ must already be a $T$-torsor in the Zariski topology.
\end{proof}

\begin{proof}[Proof of Proposition \ref{prop:p-div}]
We adapt the proof of Theorem~3.4 of \cite{altmann:06a}. Since $\CH(T)$ is free and the action of $T$ on $X$ is effective (so that each $\mathcal{A}_v$ is non-zero), there exists a (non-unique) homomorphism $\CH(T)\to k(X)^*$, $v\mapsto f_v$ satisfying $f_v\in k(X)_{v}$ for all $v\in \CH(T)$.

There exists a finite subgroup scheme $S\subset T$ containing all stabilizer groups $T_x$. Indeed, by \cite{sumihiro:74a}, it suffices to show that a linear action of a torus on $\A^n$ only admits finitely many different stabilizer groups, and this is a straightforward calculation.
Let $T'=T/S$, leading to an inclusion $\CH(T')\subset \CH(T)$. Set  $X'= X/S = \spec_Y \bigoplus_{u\in \CH(T')} \mathcal{A}_u$. Then $T'$ acts freely on $X'$, so by Lemma~\ref{lemma:free}, $X'$ is a $T'$-torsor over $X'/T'=X/T=Y$ in the Zariski topology. Equivalently, $\mathcal{A}_v$ is an invertible sheaf for $v\in \CH(T')$ and the multiplication maps $\mathcal{A}_v\otimes \mathcal{A}_{v'}\to \mathcal{A}_{v+v'}$ are isomorphisms for $v, v'\in \CH(T')$. Thus there exists a unique homomorphism $\D:\CH(T')\to \CaDiv(Y)$ such that for all $v\in \CH(T')$ the map
\[ \mathcal{A}_{v} \to k(Y), \quad f \mapsto ff^{-1}_v\in k(Y)   \]
identifies $\mathcal{A}_{v}$ with $\CO_Y(\D(v))$. Since $\CH(T)/\CH(T')$ is torsion and $\QQ$ is uniquely divisible, there exists a unique extension $\D:\CH(T)\to \CaDiv_\QQ Y$. If $f$ is a local section of $\CO_Y(\lfloor \D(v)\rfloor)$ and $n>0$ is such that $nv\in \CH(T')$, we have $f^n f_v^n= f^nf_{nv}\in \mathcal{A}_{nv}$, and hence $ff_v\in \mathcal{A}_v$ since $\mathcal{A}$ is normal. Thus, multiplication by $f_v$ defines homomorphisms
\[ \beta_v : \CO_Y(\lfloor \D(v)\rfloor) \to \mathcal{A}_{v}  \]
as desired, and it's clear that they are multiplicative. 

We check that the induced homomorphism $\beta=\bigoplus \beta_v$ is an isomorphism. This is clearly local on $Y$, so we can assume that $Y=\spec A_0$, $X=\spec A$, $A=\bigoplus_{v\in \CH(T)} A_v$. Let $B_v = f^{-1}_v A_v \subseteq k(Y)$ be the fractional ideal corresponding to $\CO_Y(\lfloor \D(v)\rfloor)$. Then $\beta$ corresponds to the map 
\[  \beta:  B=\bigoplus_{v\in \CH(T)} B_v\cdot \chi^v \to \bigoplus_{v\in \CH(T)} A_v = A . \]
which multiplies elements of $B_v$ by $f_v$.
It suffices to check that $B$ is a normal domain and that $\beta$ induces an isomorphism on fraction fields. This follows by the last paragraph of the proof of Theorem~3.4 in \cite{altmann:06a}.
\end{proof}

In general, if $H$ is a subgroup scheme of a torus $T$, passing from $X$ to $X'=(T\times X)/H$ allows us to reduce questions about $H$-varieties to questions about $T$-varieties:

\begin{lemma} \label{lemma:xprime} 
Let $\phi:H\to G$ be a homomorphism of diagonalizable group schemes whose cokernel is a torus. In the situation above, consider the $\CH(G)$-graded $\CO_Y$-algebra
\[ \mathcal{A}' = \bigoplus_{v\in \CH(G)} \mathcal{A}_{\phi^*(v)}\cdot \chi^v = (\CO_Y[\CH(G)]\otimes \mathcal{A})^H  \]
where $\CO_Y[\CH(G)]$ is given the $\CH(H)$-grading in which $\chi^v$ has weight $-\phi^*(v)$. The kernel of the map $f:\mathcal{A}'\to\mathcal{A}$ identifying $\mathcal{A}'_v = \mathcal{A}_{\phi^*(v)}\cdot \chi^v$ with $\mathcal{A}_{\phi^*(v)}$ is generated by $\chi^v - 1$ for $v\in \ker(\phi^*)$. Let $X'=\spec_Y \mathcal{A}'$, with the induced $H$-equivariant map $f:X\to X'$. Then $X'$ is a $G$-variety with an almost free action satisfying condition~\eqref{eqn:cover-condition}, and is identified by construction with the quotient $(G\times X)/H$ where $H$ acts on $G$ via the inverse of $\phi$. Moreover, for any two points $x\in X(k)$, $x'\in X'(k)$ with the same image in $Y$, the $G$-stabilizer of $x'$ is the image of the $H$-stabilizer of $x$ under $\phi$.
\end{lemma}

\begin{proof}
Self-evident.
\end{proof}

\begin{thm}\label{thm:p-div}
Let $H$ be a diagonalizable group scheme, and choose an injective homomorphism $\phi:H\to T$ into a torus $T$. As before, let $X=\spec_Y \mathcal{A}$, $\mathcal{A}=\bigoplus_{u\in M}\mathcal{A}_u$ be an $H$-variety with an almost free action satisfying condition~\eqref{eqn:cover-condition}, with quotient $\pi:X\to Y$. Let $X'=(T\times X)/H= \spec_Y \mathcal{A}'$, $\mathcal{A}'= \bigoplus_{v\in \CH(T)} \mathcal{A}_{\phi^*(v)}\cdot \chi^v$ be as in Lemma~\ref{lemma:xprime}.
\begin{enumerate}[(a)]
\item There exists a homomorphism $\D:\CH(T)\to \CaDiv_\QQ Y$ and an $H$-equivariant isomorphism 
\[ X \cong \spec_Y \left(\bigoplus_{v\in \CH(T)} \CO_Y(\lfloor \D(v) \rfloor)\cdot\chi^v\right)/(\chi^v-1\,:\, v\in\ker(\phi^*)). \]
\item Let $s:M\to \CH(T)$ be a set-theoretic section of $\phi^*$, and let $z(u,u') = s(u)+s(u')-s(u+u')$ be the associated $1$-cocycle. There exists a $1$-cocycle $g:M\times M\to k(Y)^*$ with $g_{u, u'}$ a section of $\CO_Y(\lfloor \D(z(u, u'))\rfloor)$ for all $u, u'\in M$, and an $H$-equivariant isomorphism 
\[
 X \cong \spec_Y \bigoplus_{u\in M} \CO_Y(\lfloor \D(s(u))\rfloor)\cdot\chi^u,
\]
where the multiplication on the right hand side is defined by the formula
\[ 
	(f\cdot \chi^u)(f'\cdot \chi^{u'}) = (g_{u,u'}^{-1} ff')\cdot \chi^{u+u'}.
\]
\item There exists a unique homomorphism $\bar\D:M\to \CaDiv_{\QQ/\ZZ} Y$ making the diagram
\[
	\xymatrix{
		\CH(T) \ar[r]^\D \ar[d]_{\phi^*} & \CaDiv_\QQ Y \ar[d] \\
		M \ar[r]_{\bar\D} & \CaDiv_{\QQ/\ZZ} Y
	}
\]
commute. In other words, $\D(\ker(\phi^*)) \subseteq \CaDiv_\ZZ Y$, so in particular $\D(z(u, u'))$ is integral for $u, u'\in M$.
\end{enumerate}
 \end{thm}

\begin{proof}
Apply Proposition~\ref{prop:p-div} to $X'$, taking as $S$ used in the proof a subgroup scheme of $H$, and use Lemma~\ref{lemma:xprime} to obtain the first isomorphism. For the second assertion, take the multiplicative system $v\mapsto f_v:\CH(T)\to k(X')^*$ used in the proof of Proposition~\ref{prop:p-div} and consider it as a homomorphism $\CH(T)\to k(X)^*$ with $f_v$ semi-invariant of weight $\phi^*(v)$. This makes sense because each $f_v$ is in $k(Y)\cdot \mathcal{A}'_v = k(Y) \cdot \mathcal{A}_{\phi^*(v)}$. 
Define $g_{u, u'} = f_{z(u, u')}^{-1} = f_{s(u)}^{-1}f_{s(u')}^{-1}f_{s(u+u')}$.
 Then $f \mapsto ff_{s(u)}^{-1}$ defines a homomorphism as desired, which is an isomorphism on the graded pieces. Finally, $\D(v)$ is integral for $v\in \ker(\phi^*)$ by construction (this is where we use the fact that $S\subseteq H$), which shows the last assertion.
\end{proof}

\begin{rem}
In \cite[\S 3]{altmann:12a} Altmann and Petersen construct finite covers of $\PP^1$ with abelian Galois group using so-called $A$-divisors. Such an $A$-divisor is a special instance of the map $\bar\D$ from Theorem \ref{thm:p-div} in the case $Y=\PP^1$ and $H$ a finite group scheme.
\end{rem}

\begin{rem}
The use of $X'=(T\times X)/H$ in order to understand the action of a diagonalizable group $H$ with torsion is reminiscent of the construction of the \emph{Cox sheaf} of a variety $Y$ when $\Cl(Y)$ has torsion; see \cite[\S 1.4]{coxbook} for details. 
\end{rem}
\subsection{Basic setup} \label{ss:setup}
In the rest of the article, unless stated otherwise, we fix the following setup. The base field $k$ is algebraically closed of characteristic $p>0$, $H$ is a diagonalizable group scheme over $k$ with character group $M$, $\phi:H\to T$ is an injective homomorphism into a torus $T$, $s:M\to \CH(T)$ is a set-theoretic section of $\phi^*$, and $z(u, u') = s(u)+s(u')-s(u+u')$. We consider an $H$-variety $X$ such that the almost-free locus $X^\circ$ satisfies condition~\eqref{eqn:cover-condition}, and $\pi:X^\circ\to Y$ is the quotient map. We let $X'=(T\times X^\circ)/H$ as in Lemma~\ref{lemma:xprime}, considered as a $T$-variety, with quotient map $\pi'$. Let $\mathcal{A} = \pi_* \CO_{X^\circ} = \bigoplus_{u\in M} \mathcal{A}_u$, $\mathcal{A}' = \pi'_* \CO_{X'} = \bigoplus_{v\in \CH(T)} \mathcal{A}_{\phi^*(v)}\cdot \chi^v$. We fix a homomorphism $v\mapsto f_v : \CH(T)\to k(X)$ with $f_v$ semi-invariant of weight $v$, and define 
$g_{u, u'} = f^{-1}_{z(u, u')}$.
 If $H$ itself is a torus, we can always assume that $H=T$, so that $z=0$ and $g_{u,u'}=1$. Theorem~\ref{thm:p-div} gives us $\D:\CH(T)\to \CaDiv_\QQ Y$, $\bar\D:M\to \CaDiv_{\QQ/\ZZ} Y$, and the isomorphism 
\begin{equation}\label{eqn:pdiv}
 X^\circ \cong \spec_Y \bigoplus_{u\in M} \CO_Y(\lfloor \D(s(u))\rfloor)\cdot\chi^u, \quad 
 \chi^u\cdot\chi^{u'} = g_{u, u'}^{-1} \chi^{u+u'}.
\end{equation}
This representation of $X$ induces an isomorphism of $k$-algebras 
\begin{equation}\label{eqn:sinv} 
	k(X)^{\mathrm{s-inv}} \cong \bigoplus_{u\in M} k(Y) \chi^u, \quad \chi^u\cdot\chi^{u'} = g^{-1}_{u,u'}\cdot \chi^{u+u'}   
\end{equation}
Here $k(X)^{\mathrm{s-inv}}$ is the subalgebra of $k(X)$ which is generated by the semi-invariant functions. We write $\D$ and $\bar D$ in the form 
\begin{equation}\label{eqn:d}
 \D(v) = \sum_P \alpha_P(v)\cdot P, \quad \bar\D(u) = \sum_P \bar\alpha_P(u)\cdot P 
\end{equation}
where the sums range over all prime divisors $P$ in $Y$, and homomorphisms $\alpha_P : \CH(T)\to \QQ$, $\bar\alpha_P: M\to \QQ/\ZZ$ with $\bar\alpha_P(\phi^*(v)) \equiv \alpha_P(v)$ modulo 1.

For any $\bar\alpha: M\to \QQ/\ZZ$, let $\mu(\bar\alpha)$ denote the order of $\bar\alpha$, i.e., the smallest natural number $n>0$ such that $n \cdot \bar\alpha(u) = 0$ for all $u\in M$. For a prime divisor $P\subseteq Y$, we denote by $\mu(P)$ the order of the stabilizer of a generic point of $\pi^{-1}(P)$. We denote by $\Delta$ the $\QQ$-divisor
\[ \Delta = \sum_{P} \frac{\mu(P) - 1}{\mu(P)}\cdot P  \]
on $Y$.
We let $B=X\setminus X^\circ$, and for a prime divisor $D\subseteq X$ contained in $B$, we denote by $\rho_D : M\to \ZZ$ the unique homomorphism satisfying $\nu_D(f)=-\rho_D(u)$ if $f\in k(X)^*$ has weight $u$ (cf. Lemma~\ref{lemma:boundary} below). We also define a polytope 
$$
P_X=\{u\in M_\QQ\ |\ \rho_D(u)\leq 1\}
$$
where $D$ ranges over all prime divisors $D$ contained in $B$. For $u\in M$ and $n\in\ZZ$, we will write $u\in n\cdot P_X$ meaning that the image of $u$ in $M_\QQ$ is in $n\cdot P_X$.

\begin{prop}[{cf. \cite[Corollary 7.11]{altmann:06a}}]\label{prop:stab}
In the above situation, let $P$ be a prime divisor on $Y$. Then the stabilizer of every generic point of the preimage has character group $\ker(\bar\alpha_P:M\to \QQ/\ZZ)$, and hence is isomorphic to $\mu_{n}$ where $n=\mu(\bar\alpha_P)$. In particular, $\mu(P) = \mu(\bar\alpha_P)$.
\end{prop}

\begin{proof}
We can assume that $H=T$, by replacing $X$ by $X'=(T\times X)/H$, which has the same $\D$ and stabilizers as $X$. We may shrink $Y$ until it contains no $P' \neq P$ in the support of $\D$  (i.e., $\alpha_{P'} = 0$ for $P'\neq P$). If $P$ itself is not in the support of $\D$, we see that $X$ is a $T$-torsor over $Y$. In any case, we may shrink $Y$ further so that $Y$ and $X$ are affine with coordinate rings $A_0$ and $A$, respectively, and $P$ is principal. Choosing a basis of $M$ such that all but one basis element is contained in $\ker(\alpha_P)$, we may reduce to the case $M=\ZZ$. But then the stabilizer must be of the form $\mu_{n}$, and by the proof of Proposition \ref{prop:p-div}, $n$ is exactly the smallest integer such that $n \alpha_P P$ is a $\ZZ$-divisor, that is, $n=\mu(\alpha_P)$.
\end{proof}

\begin{lemma}\label{lemma:boundary} 
Let $X$ be an $H$-variety as above, $D$ any prime divisor not intersecting the almost-free locus $X^\circ$, and $H_\redu^0$ the reduced connected component of the identity in $H$. Then the stabilizer of $H_\redu^0$ (and hence of $H$) at the generic point of $D$ is one-dimensional. If $$\rho_D\in \CH(H_\redu^0)^*= \Hom(\CH(H_\redu^0),\ZZ)=\Hom(M,\ZZ)$$ is the associated primitive co-character such that the generic point is attractive under the corresponding one-parameter subgroup, then any non-zero rational function $f$ of weight $u\in M$ vanishes to order $-\rho_D(u)$ on $D$.
\end{lemma}
\begin{proof}
Both claims follow from the proof of \cite[Proposition 3.2]{hausen:10a}.
\end{proof}

\begin{ex}[Blow up of a flag variety $F(1,1,1)$]
\label{ex:main}
  We consider the variety $W=F(1,1,1)=SL_3/B$ of complete flags in $k^3$. It is well known that $W$ is  isomorphic to the hypersurface \[W=V(x_0z_0 - x_1z_1 + x_2z_2) \subset \PP^2 \times \PP^2.\]
We denote $\GG_m^2$ by $T$ and obtain a $T$-action on $W$ given by the weight matrix
\[
\begin{array}{rrrrrrrl}
  &x_0&x_1&x_2&z_0&z_1&z_2&
\vspace{2mm}\\
 \ldelim({2}{0.5ex}
  &0&1&0&0&-1&0&\rdelim){2}{0.5ex} \\
  &0& 0&1&0&0&-1& \quad.
\end{array}
\]
It is easy to see that the locus $W^\circ$ of finite stabilizers is covered by the two open subsets $U_1 = [x_0x_1x_2 \neq 0]$ and $U_2 = [z_0z_1z_2 \neq 0]$. In particular, there are no divisors contained in $W\setminus W^\circ$. 
We have 
\[ U_1 \cong   T \times \underbrace{V(z_0-z_1+z_2)}_{\cong \PP^1}; \quad U_2\cong \underbrace{V(x_0-x_1+x_2)}_{\cong \PP^1} \times T\]
with the canonical $T$-action on the right-hand-sides. In particular the torus acts with trivial stabilizers on $W^\circ$. The quotient morphisms are both induced by 
 \[\pi:\PP^2 \times \PP^2 \dashrightarrow \PP^1;\quad (x_0:x_1:x_2,\; z_0:z_1:z_2) \mapsto (x_0z_0:x_1z_1).\]
The image of the intersection $U_1 \cap U_2$ under this quotient is $\PP^1 \setminus \{0,1,\infty\}$. Hence, $Y = W^\circ/T$ is the projective line with doubled points $0$,$1$, and $\infty$. 
Let us choose $y_0,y_1$ as coordinates for $\PP^1$. Via the embedding of function fields induced by the dominant morphism $\pi$ we have $y_0 = x_0z_0$ and $y_1 = x_1z_1$. For the structure sheaves of $U_1$ and $U_2$ we obtain
\[\CO_{U_1} = \CO_{\PP^1}[(\nicefrac{x_1}{x_0})^{\pm1}, (\nicefrac{x_2}{x_0})^{\pm1}]\]
 and
\[\CO_{U_2} = \CO_{\PP^1}[(\nicefrac{z_1}{z_0})^{\pm1}, (\nicefrac{z_2}{z_0})^{\pm1}]\]
with generators living in degrees $\pm (0,1)$ and $\pm (1,0)$. We have  $\nicefrac{y_1}{y_0}=\frac{x_1z_1}{x_0z_0}$ and using the equation $x_0z_0 - x_1z_1 + x_2z_2=0$ we obtain $\frac{x_2z_2}{x_0z_0}=\nicefrac{y_1}{y_0}-1$. This gives
\[\frac{z_1}{z_0}=\frac{y_1}{y_0} \cdot \frac{x_0}{x_1},\qquad \frac{z_2}{z_0}=\left(\frac{y_1}{y_0}-1\right)\cdot\frac{x_0}{x_2}.\]
Note, that $\div(\nicefrac{y_1}{y_0}) = [0]-[\infty]$ and $\div(\nicefrac{y_1}{y_0}-1) = [1]-[\infty]$. Hence, 
setting $\D_1(a,b) = 0$ and $\D_2(a,b)= (a+b)\cdot[\infty]-a\cdot[0]-b\cdot[1]$  gives 
\[\CO_{U_1} = \bigoplus_{u\in\ZZ^2} \CO(\D_1(u)), \qquad \CO_{U_2} = \bigoplus_{u\in\ZZ^2} \CO(\D_2(u)).\]
Since  $\D_1(u)$ and $\D_2(u)$ coincide on $\PP^1 \setminus \{0,1,\infty\}$, they induce a divisor $\D(u)$ on the non-separated prevariety $Y$ (which was covered by two instances if $\PP^1$). We obtain
\[W^\circ =U_1 \cup U_2 = \spec_Y \bigoplus_{u\in\ZZ^2} \CO(\D(u)).\]

Now, consider the one-parameter subgroup $\lambda: \GG_m \hookrightarrow T$ acting with weights 
\[
\begin{array}{rrrrrrrl}
  &x_0&x_1&x_2&z_0&z_1&z_2&
\vspace{2mm}\\
 (&0& 0&1&0&0&-1&).
\end{array}
\]
 The fixed point set of these action consists of two connected components: the lines $(0:0:1,*:*:0)$ and $(*:*:0,0:0:1)$, which are in fact both $T$-invariant. The first one contains sources and the second one contains sinks of the $\GG_m$-action, which is free in a neighborhood of these sets.
A local calculation shows that the exceptional divisors of the blowup $\widetilde W \to W$ in these lines consist of $\lambda$-fixed points, as well. 
In particular we have $\widetilde W^\circ = W^\circ$  and we then obtain two prime divisors $D_+$ and $D_-$ in $\widetilde W \setminus \widetilde W^\circ$. Lemma~\ref{lemma:boundary} implies that $\rho_{D_+} = (0,1)$ and  $\rho_{D_-} = -(0,1)$ holds. 
We obtain $P_W = M_\QQ$ and $P_{\widetilde W} = \{(a,b) \in M_\QQ \mid -1 \leq b \leq 1\}$. The boundary divisor $\Delta$ is trivial in both cases (since $\D(u)$ was integral).

We continue this example and discuss the $F$-splitting and $F$-regularity of $W$ and $\widetilde W$ in Example \ref{ex:bl1}.
\end{ex}

\begin{ex}[Cyclic covers]\label{ex:cyclic}
Set $H=\mu_n \subseteq \mathbb{G}_m = T$, and $X$ be an $H$-variety satisfying \eqref{eqn:cover-condition}. In this case, $M=\ZZ/n\ZZ$, and we choose the ``elementary school arithmetic'' section $\ZZ/n\ZZ\to \ZZ$ with image in $[0, n-1)$. Then Theorem~\ref{thm:p-div} states that
\[ X \cong \spec_Y \bigoplus_{i=0}^{n-1} \CO_Y\left(\left\lfloor \frac{i}{n} D \right\rfloor\right)   \]
for some divisor $D$ on $Y$, with multiplication of the $i$-th and $j$-th graded piece defined by the usual product if $i+j<n$, and using division (``carrying'') by a section $g$ of $\CO_Y(D)$ if $i+j\geq n$ (in which case $z(i,j)=n$). This can be seen in an elementary way if $X=\spec A$, $A=\bigoplus_{i=0}^{n-1} A_i$ is affine: let $f_1\in A_1$ be a nonzero element, and let $g = f_1^n\in A_0$. This defines a homomorphism $A_0[t]/(t^n-g)\to A$ sending $t$ to $f_1$, inducing an isomorphism of fraction fields, and hence identifying $A$ with the integral closure of $A_0$ in ${\rm Frac}(A_0)(g^{1/n})$. This also gives us maps $A_i\to {\rm Frac}(A_0)$ sending $f$ to $f/f_1^i$, and it is easily seen that the image is $\{h\in {\rm Frac}(A_0)\,:\, n\div(h) + i\div(g)\geq 0\}$. If we define $\D(i) = \frac{i}{n}\cdot \div (g)$, we now get the desired isomorphisms $\tilde A_i \cong \CO_Y(\lfloor \D(i)\rfloor)$. Moreover, we have $g_{i, j} = g$ if $i+j\geq n$, $g_{i,j}=1$ otherwise. 

Suppose that the divisor $D=\div (g)$ is reduced, so that $X= \spec A[t]/(t^n-g)$. Then $X'=\spec A[t, q, q^{-1}]/(t^n - gq)$ where $t$ has weight $1$ and $q$ has weight $n$, and the map $A[t, q, q^{-1}]/(t^n - gq)\to A[t]/(t^n-g)$ sends $q$ to $1$. The stabilizer at a point of $X'$ mapping to $D$ is $H=\mu_n$. In particular, if $n$ is divisible by the characteristic of $k$, this gives an example of a $T$-variety with a point whose stabilizer is non-reduced.
\end{ex}

\section{Preliminaries on Frobenius}\label{sec:frob}
We fix now a prime $p$, and assume that our algebraically closed field $k$ has characteristic $p$. Let $X$ be a $k$-scheme. 
By $F_X:X\to X$ (or simply $F$) we denote the \emph{absolute Frobenius} of $X$, that is, the identity map on the underlying topological space and the $p$-th power map $F^* : \CO_X \to F_* \CO_X = \CO_X$ on the structure sheaf. This means that $F_* \CO_X$ is just $\CO_X$ as a sheaf of rings, but has an $\CO_X$-module structure defined by $x\ast f = x^p f$. 

\begin{defn}  Let $X$ be a $k$-scheme.\label{defn:frob}
\begin{enumerate}[(1)]
  \item (Mehta--Ramanathan \cite{mehta:85a}, see also \cite[\S 1.1]{brion:05a}) A \emph{Frobenius splitting (or \emph{$F$-splitting})} of $X$ is an $\CO_X$-linear map $\sigma: F_* \CO_X\to \CO_X$ satisfying $\sigma \circ F^* = id$. 
  \item We say that an $F$-splitting $\sigma$ is \emph{compatible} with a closed subscheme $Z\subseteq X$ defined by a sheaf of ideals $\CI_Z$ if $\sigma(F_* \CI_Z) \subseteq \CI_Z$.
  \item ({Ramanan--Ramanathan, cf. \cite[1.4.1]{brion:05a}}) Assume $X$ is normal and let $D$ be an effective divisor on $X$, giving rise to a reflexive sheaf $\CO_X(D)$ and a section $s:\CO_X\to \CO_X(D)$. We say that a $F$-splitting $\sigma: F_{*} \CO_X\to \CO_X$ is a \emph{$D$-splitting} if it extends along $s$ to a map $F_{*} (\CO_X(D))\to \CO_X$. 
  \item Assume that $X$ is normal, and let $\Delta$ be an effective $\QQ$-divisor on $X$. By an $F$-splitting of the pair $(X, \Delta)$ we mean a $D$-splitting of $X$, where $D = \lceil (p-1) \Delta \rceil$.
\item ({\cite[Definition 3.1]{schwede:10a}}) Assume that $X$ is normal, $\Delta$ an effective $\QQ$-divisor on $X$. We say $(X,\Delta)$ is \emph{$F$-regular}\footnote{In \cite{schwede:10a}, what we call $F$-regularity is called \emph{global} $F$-regularity.} if for every effective divisor $D>0$, there exists some $e>0$ such that the map $\CO_X\to F_*^e \CO_X(\lceil (p^e-1) \Delta \rceil + D)$ splits as a map of $\CO_X$ modules. We say that $X$ is $F$-regular if $(X,0)$ is $F$-regular.
\end{enumerate}
\end{defn}

Note that if a pair $(X,\Delta)$ is $F$-regular, then it is automatically $F$-split.
The following theorem provides a useful criterion for checking $F$-regularity:

\begin{thm}[{\cite[Theorem 3.9]{schwede:10a}}]\label{thm:schwede}
The pair $(X,\Delta)$ is $F$-regular if and only if there exists an effective divisor $C>0$ on $X$ satisfying the following two properties: 
\begin{enumerate}
\item There exists an $e>0$ such that the map 
 $$\CO_X\to F_*^e \CO_X(\lceil (p^e-1) \Delta \rceil + C)$$
splits.
\item The pair $(X\setminus C, \Delta_{X\setminus C})$ is $F$-regular.
\end{enumerate}
\end{thm}

\begin{rem}\label{rem:psplit}
Suppose that a normal scheme is $D$-split for some effective divisor $D=\sum a_P P$. Then $a_P< p$ for all $P$. In particular, if $(X, \Delta)$ is $F$-split for a $\QQ$-divisor $\Delta=\sum b_P P$, then $b_P \in [0, 1]$ for all $P$. Indeed, if $D'\leq D$ is an effective divisor and $X$ is $D$-split, then it is $D'$-split as well, so the claim is that $X$ cannot be $D$-split for $D= p P$ with a single prime divisor $P$. Shrinking $X$, we can moreover assume that $P$ is Cartier. In this situation, $F_* \CO_X(D) = F_* \CO_X(pP) = F_* (F^* \CO_X(P)) = (F_* \CO_X)\otimes \CO_X(P)$ by the projection formula. Using this identification, $\CO_X\to F_* \CO_X(D)$ is the composition of the canonical section $s_P:\CO_X\to\CO_X(P)$ and $F^*\otimes id: \CO_X(P) = \CO_X\otimes\CO_X(P)\to (F_*\CO_X)\otimes\CO_X(P)= F_*\CO_X(D)$. Thus if $\CO_X \to F_* \CO_X(D)$ splits, so does $s_P:\CO_X\to \CO_X(P)$, which is impossible. 
\end{rem}

\begin{lem} \label{lem:dsplit}
  Let $X$ be an integral normal $k$-scheme, $K$ its function field, $D = \sum a_P P$ a divisor on $X$, and $\sigma_K: F_*^e K \to K$ a $K$-linear map. Denote by $\nu_P$ the valuation of $K$ of $X$ corresponding to a prime divisor $P$. Then $\sigma_K$ restricts to a map
$F_*\CO_X(D)\to \CO_X$
if and only if for all prime divisors $P$ on $X$
  \[ \nu_P(f) \geq -a_P + p^e \quad \Rightarrow \quad \nu_P(\sigma_K(f)) \geq 1 \quad \text{for all } f\in K.   \]
\end{lem}

\begin{proof} 
As $\CO_X(D)$ can be identified with the sheaf of rational functions with poles of order $\leq a_P$ along each prime divisor $P$, we see that $\sigma_K$ restricts as desired if and only if $\nu_P(f) \geq -a_P \Rightarrow \nu_P(\sigma_K(f)) \geq 0$. Since $\sigma_K(g^{p^e} f) = g\sigma_K(f)$, substituting $f{g^{p^e}}_P$ for $f$ where $\nu_P(g_P) = -1$ yields the desired result. 
\end{proof}

\begin{rem}\label{rem:duality}
When we calculate examples, it will often be convenient to relate $F$-splittings to sections of the $(p-1)$-st power of the anticanonical sheaf. Let $X$ be normal, $\Delta$ an effective $\QQ$-divisor, and $D$ any divisor on $X$. If $U\subseteq X$ is the smooth locus, the relative dualizing sheaf of $F^e:U \to U$ is $\omega_{U}\otimes (F^e)^* \omega_{U}^{-1} = \omega_{U}^{1-p^e}$.
By Grothendieck duality, we have for any $e\geq 0$ an $\CO_U$-linear isomorphism 
$$
\SHom_{\CO_U}(F_*^e \CO_U(\lceil (p^e-1)\Delta+D\rceil),\CO_U) \cong \CO_U(\lfloor(1-p^e)(K_U+\Delta)-D\rfloor),
$$
Using the $S_2$-property, we can push this isomorphism forward to $X$, see e.g. \cite[Remark 2.5]{schwede:10a}. Taking global sections, we obtain an identification
$$
\Hom_{\CO_X}(F_*^e \CO_X(\lceil (p^e-1)\Delta+D\rceil),\CO_X) \cong H^0\left(X,\CO_X(\lfloor(1-p^e)(K_X+\Delta)-D\rfloor)\right).
$$

\end{rem}
\begin{ex}\label{ex:toricsplit}
If $X$ is a toric variety defined by a fan $\Sigma$, then $-K_X$ can be chosen to be the complement of the open orbit, in which case a basis for its sections is given by monomials $\chi^{-u}$, where $u$ is a lattice point in the polytope
$$
P_X=\{u\in M_\QQ\ |\ \rho(u)\leq 1\ \forall\rho\in\Sigma{(1)}\}.
$$
Here, $M$ is the character lattice of the torus acting on $X$, $\Sigma^{(1)}$ is the set of rays of $\Sigma$, and $\rho(u)$ denotes the value of the primitive generator of $\rho$ on $u$. By the above remark, Laurent polynomials $\sum_{u\in M\cap (p-1)P_X}a_u\chi^u$ correspond to maps $F_* \CO_X\to \CO_X$. For such a map to be a splitting, the coefficient of $\chi^0$ must be equal to one; this condition is also sufficient if $X$ is complete \cite{payne:09a}. See also Lemma~\ref{lemma:invariant}.
\end{ex}

\section{Torus Actions and Frobenius}\label{sec:tfrob}

Consider the setup and notation of \S\ref{ss:setup}, and assume that $H$ has no $p$-torsion (see Remark~\ref{rem:with-p-torsion} below for what we can say without this assumption). Our main result on Frobenius splittings and $F$-regularity is the following:
\begin{thm}\label{thm:main}
The $H$-variety $X$ is $F$-split ($F$-regular) if and only if $(Y,\Delta)$ is $F$-split ($F$-regular).
\end{thm} 

We start by endowing the sheaves $F_* \CO_X$ and $\SHom(F_* \CO_X, \CO_X)$ with an $H$-equivariant structure. This is rather straight-forward, but can cause some confusion, as we work with the absolute Frobenius morphisms, which are \emph{not} morphisms of $k$-schemes. 
To remedy this, one usually introduces the relative ($k$-linear) Frobenius morphisms $F_{X/k}:X\to X'$ where $X'=X\otimes_{k, F_k} k$ is the ``Frobenius twist'' of $X$. On the other hand, in commutative algebra and in the literature on $F$-splittings and $F$-singularities, it is customary to work with the absolute Frobenius morphisms, and indeed it would be annoying to have to keep track of the various twists of everything in sight, especially since we will be interested in iterates of the Frobenius. 

Fortunately, in our situation the group $H=\spec k[M]$ is naturally defined over $\mathbb{F}_p$ (that is, we are given an $\mathbb{F}_p$-group scheme $H_0=\spec \FF_p[M]$ and an isomorphism $H\cong H_0\otimes_{\mathbb{F}_p} k$). We can now view the action of $H$ on $X$ over $k$ as an action of $H_0$ on $X$ considered as an $\mathbb{F}_p$-scheme. The Frobenius $F_{H_0}:H_0\to H_0$ is simply the multiplication by $p$ map on the group scheme, and induces the multiplication by $p$ map on $M$. From the point of view of $H_0$, an iterate of the absolute Frobenius $F^e_X:X\to X$ is $F^e_{H_0}$-equivariant. In particular, the push-forward $F^e_* \CO_X$ has a canonical $H_0$-equivariant structure, \emph{when we view $X$ as an $H_0$-scheme with $H_0$ acting via $F^e_{H_0}$}. In particular, as $\ker(F_H^e)$ acts trivially on $X$ in this action, the push-forward decomposes as $F_*^e \CO_X=\bigoplus_{u\in M/p^eM} (F_*^e \CO_X)_u$. 

If $X=\spec A$ is affine, with $A=\bigoplus_{u\in M} A_u$, then the twisted action corresponds to the grading $A=\bigoplus_{u\in M} A_{u/p^e}$, with the convention that $A_{u/p^e}=0$ if $u$ is not a multiple of $p^e$ (note the absence of $p$-torsion in $M$). The push-forward $F^e_* \CO_X$ corresponds to $A$ with the usual grading, and for $u\in M$ the $u$-graded piece of the graded module $\Hom(F^e_* \CO_X, \CO_X)$ consists of $\sigma:A\to A$
satisfying 
\[ \sigma(f^{p^e} g) = f\sigma(g) \quad \text{ and }\quad
 \sigma(A_{u'}) \subseteq  A_{(u'-u)/p^e}. \]

The isomorphism \eqref{eqn:sinv} of \S \ref{ss:setup} induces for all $u\in M$ homomorphisms of $k(Y)$-vector spaces:
\begin{equation}\label{eqn:rationaliso}
  \sigma\mapsto\bar\sigma:\Hom_{k(X)}(F_*^ek(X)^{\mathrm{s-inv}},k(X)^{\mathrm{s-inv}})_u \to  \Hom_{k(Y)}(F_*^ek(Y),k(Y)),
\end{equation}
defined by $\bar\sigma(f) = \sigma(f\cdot\chi^u)_0\in k(Y)$, the degree $0$ part of $\sigma(f\cdot \chi^u)$
with respect to the $M$ grading.

\begin{lemma} \label{lemma:rationaliso}
The homomorphisms \eqref{eqn:rationaliso} are isomorphisms.
\end{lemma}

\begin{proof}
The inverse map $\bar\sigma \mapsto \sigma$ is defined by 
\[ \sigma(f\cdot \chi^{u'}) = \sigma(f)\cdot \chi^{(u'-u)/p^e}, \]
if $u'-u \in p^e M$, and zero otherwise.
\end{proof}

The following lemma allows us to relate $F$-splittings of $X^\circ$ to $F$-splittings of $(Y, \Delta)$:

\begin{lemma}\label{lemma:piotr}
Let $E$ be any divisor on $Y$, with pullback $\widetilde{E}$ to $X^\circ$. Then the isomorphism \eqref{eqn:rationaliso} induces isomorphisms
$$\Hom_{\CO_{X^\circ}}(F_*^e\CO_{X^\circ}(\widetilde{E}),\CO_{X^\circ})_u\cong  \Hom_{\CO_Y}(F_*^e\CO_Y(\lceil (p-1)\Delta+\D(s(u)) \rceil+E),\CO_Y).$$

\end{lemma}
\begin{proof}
Assume first that $H=T$, that is, $H$ is a torus. Without loss of generality, we may assume that $Y=\spec A_0$ is affine, $X^\circ=\spec A$ is affine with $A=\bigoplus_{u\in M}A_u$, and $\D(u)=\alpha(u)P$ for some $\alpha:M\to \QQ$ and a prime principal divisor $P=V(g)$. Furthermore, we may assume that $E=\beta \cdot P$ for some $\beta\in\ZZ$.
Consider an $A$-linear map $\sigma:F_* A(\widetilde E)\to A$ of degree $u$, that is, $\sigma(g^{-\beta} A_w)\subset A_{(w-u)/p^e}$, where we put $A_w=0$ if $w \notin M$.
Such a map is determined by its restriction to $g^{-\beta}A_{u}$. Indeed, for $f\in g^{-\beta}A_{p^ew+u}$, we have $f=(f'/h)\chi^{p^ew}$ for $f'\in g^{-\beta}A_{u}$, $h\in A_0$ and
$$
\sigma(f)=\sigma\left(\frac{f'}{h}\chi^{p^ew}\right)=\frac{1}{h}\chi^w\sigma\left(f'h^{p^e-1}\right)
$$
with $\sigma$ vanishing on graded pieces not of this form.
Note that this map 
$$
\Hom_{A}(F_*^e A(\widetilde{E}),A)_u \to  \Hom_{A_0}(F_*^e g^{-\beta}A_u,A_0)
$$
is induced by the isomorphism \eqref{eqn:rationaliso}. Here, we are viewing $g^{-\beta}A_u$ as a submodule of $K$, where $K$ is the field of fractions of $A_0$.

Now, an $A_0$-linear map $\tau:g^{-\beta}A_{u}\to A_0$ extends to an $A$-linear map if and only if
\[ \tau\left( g^{-\lfloor  \alpha(p^ew+u)\rfloor-\beta} \cdot A_{0}\right) \subseteq  g^{-\lfloor  \alpha(w) \rfloor} \cdot A_0, \]
for all $w\in M$. Here we extend  $\tau$ to a map $F_*K\to K$ by localization. 
But this is equivalent to
\begin{equation}\label{cond2} 
 \nu(f) \geq -\lfloor \alpha(p^ew+u)\rfloor -\beta
\quad \Rightarrow \quad \nu(\tau(f)) \geq -\lfloor \alpha(w)\rfloor
\quad (\text{for all }w\in M),
\end{equation}
where $\nu$ is the valuation corresponding to $P$.

Consider now \eqref{cond2} for all $w\in M$ and for $f'=fg^\lambda$ as $\lambda\in \ZZ$ varies. This translates to the condition
\[ \nu(\tau(f)) \geq -\min \{ \lfloor  \alpha(w) +\lambda \rfloor \,|\, w\in M,\lambda\in\ZZ\  \nu(f) \geq - \lfloor p^e(\alpha(w) +\lambda)+\alpha(u)\rfloor-\beta ) \}. \]
But as $w$ and $\lambda$ vary, the quantity $\alpha(u)+\lambda$ (appearing here twice) traces all numbers of the form $b/\mu$ with $b\in \ZZ$, $\mu=\mu(\bar\alpha)$. We can thus rewrite the above inequality as follows:
\begin{equation}\label{eqn:ineq} \nu(\tau(f)) \geq -\min \left\{ \left\lfloor \frac{b}{\mu} \right\rfloor \,|\, b\in\ZZ, \, \nu(f) \geq - \left\lfloor \frac{p^eb+\mu (\alpha(u)+\beta)}{\mu}\right\rfloor \right\}. \end{equation}
Furthermore, the right hand side of \eqref{eqn:ineq} is
$$
-\left \lfloor
\frac{1}{\mu}
\left\lceil
\frac{
-\mu(\nu(f)+\alpha(u)+\beta)}{p^e}\right\rceil
 \right \rfloor=
\left \lceil
\frac{1}{\mu}
\left\lfloor
\frac{
\mu(\nu(f)+\alpha(u)+\beta)}{p^e}\right\rfloor
 \right \rceil
$$
so \eqref{eqn:ineq} is equivalent to requiring 
\begin{equation*}  \nu(\tau(f)) \geq \left\lceil \frac{1}{\mu} \left\lfloor \frac{\mu(\nu(f)+\alpha(u)+\beta)}{p^e} \right\rfloor \right\rceil.  
\end{equation*}

Now consider an $f$ with $0\leq \nu(f)+\alpha(u)+\beta<p^e$; we can always reduce to this case by multiplying $f$ by a monomial in $g^{p^e}$. In such a situation, the right hand side of \eqref{eqn:ineq} is at most $1$, and it equals $0$ if and only if $\nu(f)+\alpha(u)+\beta < p^e/\mu$.
 We conclude that the system of inequalities \eqref{eqn:ineq} can be reduced to 
\[ \nu(\tau(f)) \geq 1 \quad \text{if} \quad \nu(f) \geq \left\lceil \frac{p^e}{\mu} -\alpha(u)\right\rceil-\beta.  \]

On the other hand, a $K$-linear map $\tau:F_* K\to K$ restricts to an element of $\Hom_{\CO_Y}(F_*\CO_Y(\lceil (p^e-1)\Delta+\D(u) \rceil+E),\CO_Y)$ if and only if
 \[ \nu(f) \geq -\left\lceil \alpha(u)+(p^e-1)\frac{\mu-1}{\mu} \right \rceil-\beta+ p^e \quad \Rightarrow \quad \nu(\tau(f)) \geq 1 \quad \text{for all } f\in K   \]
by Lemma \ref{lem:dsplit}.
But, since $\alpha(u) \in \frac{1}{\mu}\ZZ$, we have 
$$
-\left\lceil \alpha(u)+(p^e-1)\frac{\mu-1}{\mu} \right \rceil+p^e=\left\lceil \frac{p^e}{\mu} -\alpha(u)\right\rceil
$$
and the claim follows.

To treat the general case, we first apply the above argument to $X'=(T\times X^\circ)/H$. Note that the isomorphisms \eqref{eqn:rationaliso} for $X$ and $X'$ induce identifications for $v\in \CH(T)$
\begin{align*}
 \Hom_{k(X)}(F_*^ek(X)^{\mathrm{s-inv}},k(X)^{\mathrm{s-inv}})_{\phi^*(v)}
&\cong \Hom_{k(X')}(F_*^ek(X')^{\mathrm{s-inv}},k(X')^{\mathrm{s-inv}})_v   \\
\sigma&\mapsto\sigma'
\end{align*}
which can be explicitly described as $\sigma'(f\cdot \chi^{v'}) = \sigma(f\cdot \chi^{\phi^*(v')})_{\phi((v'-v)/p^e)}$ if $v-v'\in p^e\CH(T)$, $0$ otherwise. It is clear from this description that $$\sigma\in \Hom_{\CO_{X^\circ}}(F_*^e\CO_{X^\circ}(\widetilde{E}),\CO_{X^\circ})$$ if and only if $\sigma'\in \Hom_{\CO_{X'}}(F_*^e\CO_{X'}(\widetilde{E}'),\CO_{X'})$, where $\tilde E'$ is the pull-back of $E$ to $X'$. 
\end{proof}

\begin{rem} 
Our proof of Lemma~\ref{lemma:piotr}, while direct, is perhaps not too illuminating. Let us explain why we expected Lemma~\ref{lemma:piotr} and Theorem~\ref{thm:main} to be true in the first place. In the case when $H$ is a torus, there is a relation between  $K_X$ and $K_Y$, along with a formula relating sections of their integral multiples which implicitly involves the divisor $\Delta$ \cite[\S 8.1 and 8.3]{tsurvey}. 
The relation between sections of $(1-p)K$ and $F$-splittings (Remark~\ref{rem:duality}) then suggests our main theorem. To turn this expectation into a proof, one would need to check that the identifications of \cite{tsurvey} are compatible with the Frobenius trace maps. This is the approach taken in \cite{schwede:10b} for the situation of finite covers.
\end{rem}

The next goal is to relate $F$-splittings on $X$ and $X^\circ$.

\begin{lemma}\label{lemma:extends}
We have 
$$
\Hom_{\CO_X}\left(F_* \CO_X,\CO_X\right)=\bigoplus_{u\in M\cap (p-1)P_X}\Hom_{\CO_{X^\circ}}\left(F_* \CO_{X^\circ},\CO_{X^\circ}\right)_u
$$
\end{lemma}
\begin{proof}
Consider any non-zero eigensection $\sigma_u\in \Hom_{\CO_{X^\circ}}\left(F_* \CO_{X^\circ},\CO_{X^\circ}\right)_u$. The claim is that this section extends to $X$ if and only if $u\in (p-1)P_X$. Now, a semi-invariant $f\in F_* \CO_{X^\circ}$ of weight $w$ is regular on a general point of a prime divisor $D$ from above exactly if $-\rho_D(w)\geq 0$. On the one hand, $\sigma_u(f)$ has weight $u+w$, so is regular if and only if $\sigma_u(f)$ is zero, or $-\frac{1}{p}\rho_D(w+u)\geq 0$. But $\sigma_u(f)=0$ if $u+w\notin pM$.
Hence $\sigma_u$ extends to $X$ if $u\in (p-1)P_X$. 

On the other hand, since $\sigma_u\neq0$, there locally exists a semi-invariant function $f$ of some weight $w$ such that $\sigma_u(f)\neq 0$. This implies that for any weight $w'\in w+pM$ there locally is a semi-invariant function $f'$ of weight $w'$ with $\sigma_u(f')\neq 0$. We can choose $w'$ such that $0\leq -\rho_D(w')<p$, in which case we must have $-\frac{1}{p}\rho_D(w'+u)\geq 0$, that is, $\rho_D(u)\leq -\rho_D(w')<p$.
\end{proof}

\begin{lemma}\label{lemma:invariant}
Consider a section 
$$\sigma\in \Hom_{\CO_X}\left(F_* \CO_X,\CO_X\right)$$
with decomposition $\sigma=\bigoplus_{u\in M}\sigma_u$ into eigensections. If $\sigma$ is an $F$-splitting, then so is $\sigma_0$. Conversely, if $\sigma_0$ is an $F$-splitting and $\sigma_u=0$ for all $u\in pM$, $u\neq 0$, then so is $\sigma$. Finally, $\bar \sigma_0$ is a splitting if and only if $\sigma_0$ is.
\end{lemma}
\begin{proof}
Such a section $\sigma$ is an $F$-splitting if and only if $\sigma(1)=1$. Since $1$ is an eigenfunction of weight $0$ in both $\CO_X$ and $F_* \CO_X$, $\sigma(1)=1$ implies that $\sigma_0(1)=1$ as well, hence $\sigma_0$ is an $F$-splitting. On the other hand, since $\sigma_u(1)$ has weight $u$ in $\CO_X$, and the weight of any semi-invariant function in $\CO_X$ is a multiple of $p$, we get that $\sigma(1)=\sum_{u\in pM} \sigma_u(1)$ and the second claim follows.

For the final claim, note that $\sigma_0(1)=\bar \sigma_0(1)\in k(Y)\subset k(X)$.
\end{proof}

\begin{lemma}\label{lemma:fregboundary}
If $X^\circ$ is $F$-regular, then so is $X$.
\end{lemma}
\begin{proof}
Suppose that $X^\circ$ is $F$-regular. Since $X$ is normal, the property of being $F$-regular is independent of sets of codimension at least two, and we may assume that $X$ is non-singular and $B:=X\setminus X^\circ$ is a Cartier divisor. By Theorem \ref{thm:schwede}, it suffices to show that the map 
$ \CO_X\to F_* \CO_X(B)$
splits.

Since $X^\circ$ is $F$-regular, it is $F$-split. Let $\sigma$ be any splitting, which may assume to be $H$-invariant (Lemma \ref{lemma:invariant}). Hence, $\sigma$ extends to a splitting $F_*\CO_X\to \CO_X$ (Lemma \ref{lemma:extends}). Working locally on an affine invariant chart, consider any $f \in F_*\CO_X(B)$ homogeneous of weight $u$. We must show that $\sigma(f)\in \CO_X$. Now, $f \in F_*\CO_X(B)$ implies that $\rho_D(u)\leq 1$ for any component $D$ of $B$. But then 
$$
\left\lfloor \frac{1}{p}u(\rho_D) \right\rfloor \leq 0
$$
so $\sigma(f)$ must be regular on $X$, since $\sigma(f)$ has weight $u/p$, and equals $0$ if $u/p\notin M$. Hence, $\sigma$ gives a splitting of $\CO_X\to F_* \CO_X(B)$.
\end{proof}

\begin{proof}[Proof of Theorem \ref{thm:main}]
We first deal with the statement concerning $F$-splitting.
By Lemma \ref{lemma:invariant}, if $X$ has an $F$-splitting, it has an invariant $F$-splitting. By Lemma \ref{lemma:extends}, $X$ has an invariant $F$-splitting if and only if $X^\circ$ has an invariant $F$-splitting. Finally, $X^\circ$ has an invariant $F$-splitting if and only if $(Y,\Delta)$ has an $F$-splitting by Lemma \ref{lemma:piotr} applied in the case $u=0$, $E=0$.

We now deal with $F$-regularity.
By Lemma \ref{lemma:fregboundary}, we may assume that $X=X^\circ$. Firstly, assume that $X$ is $F$-regular, and let $D$ be an effective divisor on $Y$. Then there is a splitting of $\CO_X\to F_*^e \CO(\widetilde D)$ which we may assume to be $H$-invariant (cf. Lemma \ref{lemma:invariant}) which leads to a splitting of $\CO_Y\to F_*^e\CO_Y(\lceil (p^e-1)\Delta\rceil+D)$ by Lemma \ref{lemma:piotr}. Hence, $(Y,\Delta)$ is $F$-regular.

Conversely, assume that $(Y,\Delta)$ is $F$-regular. Since $X$ and $Y$ are normal, we may remove a set of codimension at least two to arrive at the situation that $Y=U\cup C$ for some effective divisor $C$ and some non-singular affine $U$ over which $X$ is a torsor. Now, since $(Y,\Delta)$ is $F$-regular, the map $\CO_Y\to F_*^e\CO_Y(\lceil (p^e-1)\Delta\rceil+C)$ splits for some $e$, so by Lemma \ref{lemma:piotr} the map $\CO_X\to F_*^e \CO( \widetilde{C})$ splits as well, where $\widetilde{C}$ is the preimage of $C$ in $X$. Furthermore,  $X\setminus \widetilde{C}$  is affine, and non-singular since it is an $H$-torsor over $U$ and $H$ is smooth \cite[\S 4]{milne:80a}.  It follows that $X\setminus\widetilde{C}$ is $F$-regular \cite[Remark 3.3]{schwede:10a}. Hence, by Theorem \ref{thm:schwede}, $X$ is $F$-regular.
\end{proof}
\begin{rem}\label{rem:bijection}
Our proof of Theorem \ref{thm:main} actually shows that 
$H$-invariant $F$-splittings of $X$ are in bijection with $F$-splittings of $(Y,\Delta)$. Furthermore, combining Lemmas \ref{lemma:extends} and \ref{lemma:piotr} gives a graded isomorphism
$$
\Hom_{\CO_X}\left(F_* \CO_X,\CO_X\right)\cong \bigoplus_{u\in M\cap (p-1)P_X} \Hom_{\CO_Y}(F_*\CO_Y(\lceil (p-1)\Delta+\D(u) \rceil),\CO_Y).$$
Lemma \ref{lemma:invariant} providing a sufficient criterion for a section $\sigma$ of the right hand side to correspond to a splitting.
 \end{rem}

\begin{rem} \label{rem:with-p-torsion}
Many of the statements above continue to hold if we allow $H$ to have $p$-torsion. Note that in this generality, if $X=\spec A$, $A=\bigoplus_{u\in M} A_u$, the twisted action (using the $e$-th Frobenius on $H$) on $X$ corresponds to the grading 
\[ A = \bigoplus_{u\in M} A'_u \quad \text{where} \quad A'_u = \bigoplus_{w\,:\,p^ew = u} A_u, \]
and $\Hom(F^e_* \CO_X, \CO_X)_u$ consists of those $\sigma:A\to A$ which satisfy $\sigma(f^{p^e} g) = f\sigma(g)$ and $\sigma(A_{u'}) \subseteq A'_{u'-u}$, i.e., if $f\in A_u$ then $\sigma(f)_{w} = 0$ unless $p^e w = u'-u$. In this case, the map \eqref{eqn:rationaliso} is defined as $\bar\sigma(f) = \sigma(f\cdot \chi^u)_0$, the degree $0$ part of $\sigma(f\cdot \chi^u)$ with respect to the original grading on $A$ (note that $A'_0$ itself is graded by $M[p^e]$). Lemma~\ref{lemma:rationaliso} is still true, with the inverse map $\bar\sigma\mapsto\sigma$ given by the more complicated formula
\begin{equation} \label{eq:complicated} \sigma(f\cdot \chi^u) = \sum_{w\,:\, p^e w = u'-u } \bar\sigma(\lambda_{u,w} f)\cdot \chi^{w}, 
\end{equation}
where 
\[ \lambda_{u, w} = \frac{\chi^{u+p^ew}}{\chi^u(\chi^w)^{p^e} } = g_{u, p^ew} g_{(p^e-1)w,w}g_{(p^e-2)w,w}\ldots g_{2w,w}g_{w,w}\in k(Y). \]

Moreover, Lemmas~\ref{lemma:piotr} and \ref{lemma:extends} continue to hold, as does the first statement of Lemma \ref{lemma:invariant}. Furthermore, any invariant $F$-splitting $\sigma$ of $X$ induces an $F$-splitting $\bar\sigma$ of $(Y,\Delta)$. Hence, $X$ $F$-split (or $F$-regular) implies the same for $(Y,\Delta)$. The problem with the other direction in Theorem~\ref{thm:main} is that $\sigma:F_*\CO_X\to\CO_X$ (of weight $u=0$) does not have to be a splitting if $\bar\sigma$ is, as Example~\ref{example:with-p-torsion} below shows. In fact, $\sigma$ is a splitting if and only if $\bar\sigma$ is a splitting satisfying $\bar\sigma((\chi^u)^{p^e})=0$ for every $u\in M[p^e]$. In more intrinsic terms, this condition is equivalent to $\bar\sigma(f)=0$ for every $f\in k(Y)$  which is not a $p$-th power but which becomes a $p$-th power in $k(X)$. We do not know if Theorem \ref{thm:main} still holds if $H$ has $p$-torsion.
\end{rem}

\begin{ex} \label{example:with-p-torsion}
Let $H = \mu_p$, $X=\mathbb{A}^1_k$  with coordinate $x$ and the standard $\mu_p$-action. Then $Y=\mathbb{A}^1_k$ with coordinate $y$, and $\pi^* y = x^p$. In this case, for $u=0$ and $f=1$, the formula \eqref{eq:complicated} simplifies to
\[ \sigma(1) = \sum_{i=0}^{p-1} \bar\sigma(y^i)x^i.  \]
In particular, $\sigma$ is a splitting if and only if $\bar\sigma(y^i) = 0$ for $0<i<p$, $\bar\sigma(1)=1$, that is, if $\bar\sigma$ is $\mathbb{G}_m$-invariant for the standard action of $\mathbb{G}_m$ on $Y$. 
\end{ex}

\section{Compatible Splittings}\label{sec:compat}
We again consider an $H$-variety $X$, and use notation established in \S \ref{ss:setup}, assuming again that $H$ has no $p$-torsion.  We now establish two results concerning compatible splittings. Recall that $H_\redu^0$ is the reduced connected component of the identity in $H$.

\begin{prop}\label{prop:comp1}
Consider an $H_\redu^0$-invariant splitting $\sigma\in \Hom_{\CO_X}(F_*\CO_X,\CO_X)$. Then $\sigma$ is compatible with $B:=X\setminus X^\circ$, that is, $\sigma(F_* \I_B)= \I_B$. In particular, any $H$-invariant splitting is compatible with $B$.
\end{prop}

\begin{proof}
To begin, assume that $H=H^0_\redu$.
Without loss of generality, $X$ is affine with coordinate ring $A=\bigoplus_{u\in M} A_u$. 
Let $\omega$ be the cone in $M_\QQ$ generated by those $u\in M$ with $A_u\neq 0$, and $\omega'$ the face of elements invertible in the monoid $\omega$. Then the ideal $I_B$ of $B$ is given by $\bigoplus_{u\in M\cap(\omega\setminus \omega')} A_u$. Since $\sigma$ is $H$-invariant, it maps homogeneous elements of degree $u$ to degree $u/p$, so $\sigma(F_*(I_B))=I_B$.

Now to conclude the proof note that for general $H$, any $H$-invariant splitting is also $H_\redu^0$-invariant.
\end{proof}

\begin{prop}\label{prop:comp2}
Now let $S$ be any closed subscheme of $Y$ and $$\sigma\in \Hom_{\CO_X}(F_*\CO_X,\CO_X)$$ an $F$-splitting of $X$.
\begin{enumerate}
\item If the splitting $\sigma$ is compatible with $\overline {\pi^{-1}(S)}\subset X$, then its $H$-invariant part $\sigma_0$ is also compatible with $\overline{ \pi^{-1}(S)}\subset X$.
\item Suppose the $H$-invariant part $\sigma_0$ is compatible with $\overline{ \pi^{-1}(S)}\subset X$. Then $\bar\sigma_0\in\Hom_{\CO_Y}(F_*\CO_Y,\CO_Y)$ is compatible with $S$. 
\item Conversely, suppose that no component of $S$ is contained in the support of $\Delta$ and $\bar\sigma_0$ is compatible with $S\subset Y$. Then $\sigma_0$ is compatible with $\overline{ \pi^{-1}(S)}$.
\end{enumerate}
\end{prop}
\begin{proof}
By \cite[Lemma 1.1.7]{brion:05a} we may assume that $X=X^\circ$ and that $X$ and $Y$ are affine with respective coordinate rings $A_0$ and $A=\bigoplus_{u\in M} A_u$. Let $I_S\subset A_0$ be the ideal of $S$; then the ideal of $\pi^{-1}(S)=\overline {\pi^{-1}(S)}$ is $A\cdot I_S=\bigoplus_{u\in M} A_u\cdot I_S$. 
Let $\sigma=\sum_{u\in M} \sigma_u$ be the isotypical decomposition of $\sigma$. 

First, assume that $\sigma$ is compatible with $\overline {\pi^{-1}(S)}$. Consider any $f\in F_* (A\cdot I_S)$, without loss of generality homogeneous of degree $w$. Then $\sigma(f)\in A\cdot I_S$, and so we have that $\sigma_u(f)\in A_{(w-u)/p} \cdot I_S$. In particular, $\sigma_0(f)\in A_{w/p} \cdot I_S\subset A\cdot I_S$, so $\sigma_0$ is compatible with $\overline {\pi^{-1}(S)}$.

Now if $\sigma_0$ is compatible with $\overline {\pi^{-1}(S)}$, then for any degree zero element $f\in F_* (A_0 \cdot I_S)=F_*(I_S)$, we have $\sigma_0(f)\in I_S$, so $\bar\sigma_0\in\Hom_{\CO_Y}(F_*\CO_Y,\CO_Y)$ is compatible with $S$. 

On the other hand, suppose that $\bar\sigma_0(I_S)= I_S$, and that $S$ is contained in the support of $\Delta$. Then again by \cite[Lemma 1.1.7]{brion:05a}, we may shrink $Y$ and only consider the case that  $\D$ is trivial, that is, $A=\bigoplus_{u\in M} A_0\cdot \chi^u$. But then for $f\in F_*(I_S\cdot A_0\cdot \chi^u)$, $\sigma_0(f)\in  I_S\cdot A_0\cdot \chi^{u/p}$ as desired, where $\chi^{u/p}=0$ if $u/p\notin M$.
\end{proof}

\section{Separations}\label{sec:sep}
We have been able to characterize $F$-regularity and the existence of an $F$-splitting for an $H$-variety $X$ in terms of the quotient pair $(Y,\Delta)$ in Theorem \ref{thm:main}. However, the quotient $Y$ need not in general be separated. We now describe how to replace the pair $(Y,\Delta)$ with a pair $(\Ysep,\Deltasep)$ such that $\Ysep$ is separated. 
Recall that an open subscheme $U\subset Y$ is \emph{big} if $\codim_Y (Y\setminus U)>1$.
\begin{defn}
A \emph{separation} of a $k$-scheme $Y$ is rational map $\s:Y \dashrightarrow \Ysep$, such that
\begin{enumerate}
\item $\Ysep$ is separated.
\item The map $\s$ is defined on a big open subset $U \subset Y$ which maps locally isomorphically to a big open subset of $\Ysep$.
\end{enumerate}
\end{defn}

Recall that a \emph{prevariety} is an integral scheme of finite type over $k$. We will use the following proposition to replace our quotient $Y=X^\circ/T$ with a variety.

\begin{prop}\label{prop:sep}
  Every normal prevariety admits a separation.
\end{prop}

\begin{rem}
  In \cite{hausen:10a}, separated quotients of $T$-varieties are produced by considering the inverse limit of GIT-quotients. In this setting, the image of the quotient map into the GIT-limit gives a separation of $X^\circ/T$ and the distinguished component of the limit which contains the image coincides with the Chow-quotient introduced in \cite{altmann:06a}.
\end{rem}

To prove the proposition, we need several facts about centers of valuations.

\begin{defn}
 Consider a valuation $\nu$ of $k(Y)$. A center of $\nu$ is an irreducible closed subset $C \subset Y$ such that $\CO_{C, Y}\subseteq \CO_\nu$ and $\CO_{C,Y} \hookrightarrow \CO_\nu$ is a local ring homomorphism.
\end{defn}

\begin{lem}
  A prevariety $Y$ is separated if and only if every valuation of $k(Y)$ has at most one center.
\end{lem}
\begin{proof}
See \cite[Theorem 4.3]{hartshorne:77a} 
\end{proof}

\begin{lem}
  \label{lem:image-center}
  Consider dominant morphism $\phi:Y \rightarrow Y'$  from a prevariety $Y$ to a prevariety $Y'$ and a valuation $\nu$ of $k(Y)$ with center $C \subset Y$. Then $C':=\overline{\phi(C)}$ is a center of $\nu|_{k(Y')}$.
\end{lem}
\begin{proof}
Note that we have a local ring homomorphism $\CO_{Y',C'} \hookrightarrow \CO_{Y,C}$ induced by $\phi$ and a local ring homomorphism $\CO_{Y,C} \hookrightarrow \CO_\nu$ by the definition of a center. Hence, the composition is a local ring homomorphism as well. 
\end{proof}

\begin{defn}
A \emph{multiple center} of a prevariety $Y$ is a closed subset $C\subset Y$ which is the center of some valuation $\nu$, such that $\nu$ has more than one center.
  We define the \emph{non-separated locus} of a prevariety $Y$ to be the union of all multiple centers.
\end{defn}

\begin{lem}
  The locus of non-separateness of a prevariety $Y$ is a Zariski closed subset.
\end{lem}
\begin{proof}
    We consider some open affine covering $\{U_i\}_{i\in I}$ of $Y$ and denote $U_i \cap U_j$ by $U_{ij}$. We set $A_i = \CO_Y(U_i)$ and denote the sub-algebra of the function field generated by $A_i$ and $A_j$ by $A_iA_j$ and its spectrum by $\widetilde{U}_{ij}$. We have a birational map $\phi_{ij}: \widetilde{U}_{ij} \dashrightarrow {U}_{ij}$ and a commutative diagram as follows
$$
\xymatrix{
& \widetilde{U}_{ij}\ar@{-->}^{\phi_{ij}}[dd] \ar^{f_{ij}}[rd] \ar_{f_{ji}}[ld] & \\
U_i & & U_j\\
& {U}_{ij}  \ar@{_{(}->}[lu] \ar@{^{(}->}[ru]&
}
$$

Now, we denote the indeterminacy locus of $\phi_{ij}$ by $V_{ij}$. We claim that 
\[\bigcup_{i,j \in I} \overline{f_{ij}(V_{ij})} \subset Y\]
equals the locus of non-separateness.

Assume we have a point $y$ in this finite union. This means there is a pair $(i,j)$ and component of the $C$ of $V_{ij}$ such that $y \in \overline{f_{ij}(C)}$. Now, we may choose a valuation $\nu$, which has center $C$. This implies that $f_{ij}(C)$ and $f_{ji}(C)$  are centers of $\nu$, as well. Since $C$ lies in the locus of indeterminacy, $f_{ij}(C)$ and $f_{ji}(C)$  do not intersect $U_{ij}$. Hence, $\overline{f_{ij}(C)} \neq \overline{f_{ji}(C)}$. Hence, $y$ lies in the locus of non-separateness.

Assume instead that we have a point $y$ in the non-separated locus. This means it belongs to some multiple center $V$ of some valuation $\nu$. Hence, we have another center $V'$ of the same valuation. They cannot both intersect the same affine chart, since affine varieties are separated. Hence, we have two charts $U_i$ and $U_j$, such that $U_i \cap V' = U_j \cap V = \emptyset$ but $U_i \cap V \neq \emptyset$ and $U_j \cap V' \neq \emptyset$. In particular $\nu$ has no center on the intersection $U_{ij}$. The fact that $\nu$ has a center on $U_i$ and $U_j$ is equivalent to the inclusions of the coordinate ring $A_i,A_j \subset \CO_\nu$. But then we have $A_iA_j \subset \CO_{\nu}$ as well. Hence, $\nu$ has a center $C$ on $\widetilde{U}_{ij}$ and we have $\overline{f_{ij}(C)} = V$ and $\overline{f_{ij}(C)} = V'$, by Lemma \ref{lem:image-center}.  Since $\nu$ has no center on $U_{ij}$, it follows that $C$ is contained in the indeterminacy locus $V_{ij}$ of $\phi_{ij}$. Hence, $V$ and $V'$ are contained in $\overline{f_{ij}(V_{ij})}$ and $\overline{f_{ji}(V_{ij})}$ respectively.
\end{proof}

\begin{proof}[Proof of Proposition \ref{prop:sep}]
Consider the non-separated locus inside the prevariety $Y$. From the components of codimension $1$, several components have the same local ring. For every one of the local rings occurring, choose one of these components and remove the rest. The remaining prevariety $Y'$ is ``separated in codimension one'', i.e. the non-separated locus $V$ of $Y'$ has codimension $>1$. If we remove this locus from $Y'$ we obtain a variety $\Ysep=Y'\setminus V$. Now, the rational map $\s$ is just the inverse of the inclusion $\Ysep \hookrightarrow Y$. Let $D$ be a prime divisor in $Y$. Then by construction, there is a prime divisor $D' \subset \Ysep$  with the same local ring as $D$. In other words, for every prime divisor in $Y$ there is an open subset intersecting $D$ which is mapped isomorphically to an open subset of $\Ysep$.
\end{proof}
\begin{rem}
If $Y$ is a smooth prevariety of dimension one, then it admits a unique separation $\s:Y\to \Ysep$, and $\s$ is a morphism. 
\end{rem}

Consider a separation $\s:Y\dashrightarrow \Ysep$. For any $\QQ$-divisor $D=\sum_{P\subset Y}a_P\cdot P$ on $Y$, we define 
$$
\smax D :=\sum_{P'\subset \Ysep}\max\{a_P\ |\ P\subset \s^{-1}(P')\}\cdot P'.
$$
 With this we have $\s_*\CO(-D) = \CO(-\smax D)$.

A separation of a pair $(Y,\Delta)$ consists of a separation $\s:Y \dashrightarrow  \Ysep$ along with the $\QQ$-divisor $\Deltasep:=\smax\Delta$ on $\Ysep$.
By Proposition \ref{prop:sep}, such a separation always exists, although it is not necessarily unique.

\begin{rem}
Note that $\Deltasep$ is the unique minimal divisor on $\Ysep$ such that $\s^*(\Deltasep) \geq \Delta$.
\end{rem}

We may use the following result, coupled with Theorem \ref{thm:main} and Proposition \ref{prop:comp2} to characterize (compatible) $F$-splittings and $F$-regularity of a $T$-variety in terms of properties of a separated quotient.
\begin{thm}\label{thm:sep}
  Consider a normal pair $(Y,\Delta)$ and a separation $\s:(Y,\Delta) \dashrightarrow (\Ysep,\Deltasep)$. 
\begin{enumerate}
\item\label{claim:one} The pair $(Y,\Delta)$ is $F$-split ($F$-regular) if and only if $(\Ysep,\Deltasep)$ is $F$-split ($F$-regular). 
\item\label{claim:two} Let $S\subset Y$ be a closed subscheme such that $\overline{U\cap S}=S$, where $U$ is the open subset of $Y$ on which $\s$ is regular. Then  $(Y,\Delta)$ is $F$-split is compatible with $S$ if and only if $(\Ysep,\Deltasep)$ is $F$-split compatible with $\overline{\s(S)}$.
\end{enumerate}
\end{thm}

\begin{proof}
Due to our normality assumption, we may without loss of generality assume that $\s$ is regular on all of $Y$ with image $\Ysep$. Now, since $\s$ is a local isomorphism between the pairs $(\Ysep,\Deltasep)$ and $(Y,\s^*\Deltasep)$, we have an isomorphism between
$$
\Hom_{\CO_\Ysep}(F_*^e\CO_\Ysep(\lceil (p^e-1)\Deltasep\rceil+D),\CO_\Ysep)
$$
and 
\begin{equation}\label{eqn:yhom}
\Hom_{\CO_Y}(F_*^e\CO_Y(\lceil (p^e-1)\s^*(\Deltasep)\rceil+\s^*(D)),\CO_Y)
\end{equation}
for any effective divisor $D$ on $Y$ which preserves the property of being a splitting.
Furthermore, \eqref{eqn:yhom} is equal to 
\begin{equation*}
\Hom_{\CO_Y}(F_*^e\CO_Y(\lceil (p^e-1)\Delta\rceil+\s^*(D)),\CO_Y)
\end{equation*}
by definition of $\Deltasep$.
This proves claim \ref{claim:one} with regards to $F$-splitting. For $F$-regularity, note that $\s^*(\s_*(C))\geq C$ for any divisor $C$ on $Y$, and the claim follows.
Claim \ref{claim:two} follows immediately from the above isomorphism and \cite[Lemma 1.1.7]{brion:05a}.
\end{proof}

We are now going to reformulate the description of splittings of $X^\circ$ (Lemma \ref{lemma:piotr}) in terms of $\Ysep$. Remember from \S\ref{ss:setup}  that $X^\circ$ has a description as \[X^\circ \cong \spec_Y \bigoplus_{u\in M} \CO\left(\lfloor \D(s(u))\rfloor\right)\cdot\chi^u\]
where $\D:\CH(T)\to\CaDiv_\QQ Y$.
We define $h:M\to \CaDiv_\QQ \Ysep$ by $h(u) = \smax((p-1)\Delta+\D(s(u)))$. 
If $H$ is a torus, i.e. $M$ is torsion-free and $s$ is the identity, we may view $h$ as a convex and piecewise linear function on $M_\QQ$.

\begin{lem}
  \label{lemma:homs-separated} For $u \in (p-1)P_X$ we have
  $$\Hom_{\CO_X}(F_*\CO_X,\CO_X)_u\cong  \Hom_{\CO_Y}(F_*\CO_\Ysep(\lceil h(u) \rceil),\CO_\Ysep).$$
\end{lem}
\begin{proof}
As in the proof of Theorem~\ref{thm:sep} we have 
$$
\Hom_{\CO_\Ysep}(F_*\CO_\Ysep(\smax(\lceil(p-1)\Delta\rceil+\D(s(u)))),\CO_\Ysep)
$$
being equal to 
\[
\Hom_{\CO_Y}(F_*\CO_Y(\lceil (p-1)\Delta\rceil+\D(s(u))),\CO_Y). \qedhere
\]
\end{proof}

\begin{ex}[Blowup of a flag variety (continued)]\label{ex:bl1}
 For $W^\circ = \widetilde W^\circ$ from Example~\ref{ex:main} we had as a non-separated quotient the projective line with doubled points $\{0,1,\infty\}$. The separation is just the ordinary $\PP^1$ and the morphism $W^\circ \rightarrow Y \to \Ysep=\PP^1$ is again given by 
  \[\pi:\PP^2 \times \PP^2 \dashrightarrow \PP^1;\quad (x_0:x_1:x_2,\; z_0:z_1:z_2) \mapsto (x_0z_0:x_1z_1).\]

 The piecewise linear function $h:P_W \rightarrow \Div_\QQ \PP^1$ defined by $h(u):=\smax \D(u)$ is given by
\[ h(a,b) =  \max\{-a,0\}[1] \;+\; \max\{-b,0\}[1] \;+\; \max\{a+b,0\}[\infty].\]
For $\widetilde W$ we obtain just the restriction $h|_{P_{\widetilde W}}$.

Since $(\Ysep,\Deltasep) = (\PP^1,0)$ for $W$ and $\widetilde W$ we deduce by Theorem~\ref{thm:main} and Theorem~\ref{thm:sep}
that both varieties are $F$-regular (and hence $F$-split) for every prime $p$.

We continue our discussion in Example \ref{ex:bl2}, showing that both varieties are \emph{diagonally split}.
\end{ex}

\section{Special Cases}\label{sec:special}
In this section, we consider some special cases and examples of $H$-varieties where criteria for $F$-splitting and $F$-regularity simplify.
\subsection{Cyclic Covers}\label{sec:cyclic}
Let $X$ be a normal $n$-fold cyclic cover of a normal variety $Y$ with reduced branch divisor $D$, and assume that $n$ is relatively prime to $p$. Let $\Delta$ be the boundary divisor as in \S \ref{ss:setup}. 
Then by Theorem \ref{thm:main}, we have that $X$ is $F$-split ($F$-regular) if and only if 
$(Y,\Delta)$ is $F$-split ($F$-regular).
Note that the support of $\Delta$ is exactly $D$, and $\Delta$ is of the form
\[\Delta=\sum_i \frac{n_i-1}{n_i}D_i,\]
where the $D_i$ are the irreducible components of $D$ and each $n_i$ divides $n$.
If the ramification index of every point $x\in X$ in the ramification locus is equal to $n$, then we simply have \[\Delta=\frac{n-1}{n}D.\]
Note that our result for cyclic covers is simply a special case of \cite{schwede:10b}, which gives criteria for $F$-splitting and $F$-regularity to preserved under arbitrary finite morphisms with tame ramification.

\begin{prop}\label{prop:dc}
Let $X$, $Y$, $\Delta$ and $D$ be as above with $\Delta=\frac{n-1}{n}D$ and $X$, $Y$ projective.
Suppose that $\CO_Y((n-1)D)\cong \omega_Y^{-n}$.
Then $X$ is $F$-split if and only if:
\begin{enumerate}
\item  We have $p\equiv 1 \pmod n$, that is, $p-1=\alpha n$ for some $\alpha\in \NN$; and\label{crit:one}
\item The isomorphism $\CO_Y((n-1)D)\cong \omega_Y^{-n}$ induces a non-zero map $$\phi:\CO_Y(\alpha(n-1)D)\to \omega_Y^{1-p},$$ and a multiple of $\phi(1)$ corresponds to an $F$-splitting of $Y$ under Grothendieck duality.\label{crit:two} 
\end{enumerate} 
Furthermore,  $X$ is never $F$-regular.
\end{prop}

\begin{proof}
By Theorem \ref{thm:main}, we are checking whether or not the pair $(Y,\frac{n-1}{n}D)$ is $F$-split (or $F$-regular). By Remark \ref{rem:duality}, maps $F_*\CO_Y(\lceil (p-1) \frac{n-1}{n}D \rceil)\to \CO_Y$ are given by sections of
\begin{equation*}
\mcL=\omega^{1-p}_Y\left(-\left\lceil (p-1) \frac{n-1}{n}D\right\rceil\right).
\end{equation*}
Now, $\mcL^n$ is a sub-bundle of 
\begin{equation}\label{eqn:ll}
\omega^{n(1-p)}\left(-(p-1)(n-1)D\right)\cong \CO_Y
\end{equation}
with equality if and only if $(p-1) \frac{n-1}{n}D$ is an integral divisor, that is $p\equiv 1 \pmod n$. 
Since $Y$ is projective, $H^0(Y,\mcL)$ is at most one-dimensional, and must vanish unless
 $p\equiv 1 \pmod n$. Hence,  condition \eqref{crit:one} must hold for $X$ to be $F$-split.

If $\mcL$ has no sections, then $X$ is not $F$-split; assume instead that the space of global sections is generated by some non-zero $f\in H^0(Y,\mcL)$. Then $f^n\in H^0(Y,\mcL^n)$ corresponds to the isomorphism  $\CO_Y(((p-1) (n-1)D)\cong \omega_Y^{(1-p)n}$ via \eqref{eqn:ll}, and $f$ induces a map $\phi$ as in condition \eqref{crit:two}. 
Hence, assuming condition \eqref{crit:one}, condition \eqref{crit:two} is necessary and sufficient for $F$-splitting. Furthermore,  $X$ is never $F$-regular, since again by duality,
$$
\Hom_{\CO_Y}\left(F_*^e \CO_Y\left(\left\lceil (p^e-1)\frac{n-1}{n}D\right\rceil+E\right),\CO_Y\right)=0
$$
for any non-trivial effective divisor $E$.
\end{proof}

\begin{ex}[Elliptic curves as double covers]\label{ex:elliptic}
Let $X$ be a smooth elliptic curve, and $p>2$. Then an affine model of $X$ can be given by 
$$
y^2=x(x-1)(x-\lambda), 
$$
for $\lambda\neq 0,1$ which realizes $X$ as a double cover of $\PP^1$ with branch divisor $D=\{0\}+\{1\}+\{\lambda\}+\{\infty\}$.
The curve $X$ is $F$-split if and only if it is ordinary \cite[1.3.9]{brion:05a}, and it is classically known that this is the case if and only if the coefficient of 
coefficient of $x^{(p-1)/2}$ in $(x-\lambda)^{(p-1)/2}(x-1)^{(p-1)/2}$ is non-zero \cite[Corollary 4.2]{hartshorne:77a}. 

We can easily recover this result using Proposition \ref{prop:dc}. Indeed, taking $1,x$ as a basis of $\CO(1)$ with $(1)_0=\{\infty\}$ and $(x)_0=\{0\}$, we have an isomorphism $\CO(D)\to \CO(4)$ sending $1$ to $x(x-1)(x-\lambda)$. The section $\phi(1)=(x(x-1)(x-\lambda))^{(p-1)/2}\in\CO(2(p-1))$ corresponds to a splitting of $\PP^1$ if and only if the coefficient of $x^{p-1}$ is non-zero, cf. Example \ref{ex:toricsplit}. But this is the same as requiring that the coefficient of $x^{(p-1)/2}$ in $(x-\lambda)^{(p-1)/2}(x-1)^{(p-1)/2}$ is non-zero.
\end{ex}

\begin{defn}\label{defn:ordinary}
Based on the above example, we say that the pair $(\PP^1,\frac{1}{2}(c_1+c_2+c_3+c_4))$ is \emph{ordinary} if and only if the coefficient of $x^{(p-1)/2}$ in $(x-\lambda)^{(p-1)/2}(x-1)^{(p-1)/2}$ is non-zero, where $\lambda$ is the cross-ratio $(c_1,c_2;c_3,c_4)$. By the above, the pair $(\PP^1,\frac{1}{2}(c_1+c_2+c_3+c_4))$ is $F$-split if and only if it is ordinary.
\end{defn}

\begin{ex}[Elliptic curves as triple covers]\label{ex:trielliptic}
In the situation of Proposition \ref{prop:dc}, we can also take $Y=\PP^1$, $n=3$, $D=\{0\}+\{1\}+\{\infty\}$. The curve $X$ is a triple cover of $\PP^1$, and is a smooth elliptic curve as long as $p>3$. By Proposition \ref{prop:dc}, $X$ is $F$-split if and only if $p\equiv 1 \pmod 3$. Indeed, 
we again have an isomorphism $\CO(D)\to \CO(3)$ sending $1$ to $(x(x-1))$
where we take a basis of $\CO(1)$ as in Example \ref{ex:elliptic}.
Then $\phi(1)\in\CO(2(p-1))$ is given by $x^{2(p-1)/3}(x-1)^{2(p-1)/3}$, and the coefficient of $x^{(p-1)}$ is clearly non-zero.
\end{ex}

\begin{ex}[K3 double covers]
Let $Y=\PP^2$, $D$ be a smooth sextic curve in $Y$, and $X$ a double cover of $Y$ ramified over $D$. Then $X$ is a smooth K3 surface.
If $f\in k[x,y,z]$ is a sextic polynomial such that $D=V(f)$ and $p>2$, then Proposition \ref{prop:dc} implies that $X$ if $F$-split if and only the coefficient of $(xyz)^{p-1}$ in $f^{p-1}$ is non-zero.

More generally, similar statements can be made for K3 surfaces arising as double covers of smooth toric surfaces $Y$. Indeed, let $P_Y$ be as in Example \ref{ex:toricsplit}. A smooth section $f$ of $\omega_Y^{-2}$ can be written as 
$$
f=\sum_{u\in 2P\cap M} a_u\chi^u
$$
and for $p>2$ the corresponding K3 double cover is $F$-split if and only if the constant term of $f^{p-1}$ is non-zero.
\end{ex}

\begin{ex}[A Fano threefold]
Consider a homogeneous quartic polynomial $f\in k[x,y,z,w]$, $\chara k>2$, and let $X$ be a double cover of $\PP^3$ with branch locus $V(f)$. Then $X$ is a smooth Fano threefold of degree $16$, which is $F$-split if and only if $(\PP^3,\frac{1}{2}V(f))$ is $F$-split by Theorem \ref{thm:main}. By Remark \ref{rem:duality} and Example \ref{ex:toricsplit}, $(\PP^3,V(f))$ is $F$-split if and only if the coefficient of 
$x^\alpha y^\beta z^\gamma w^\delta$
in $f^{(p-1)/2}$ is non-zero for some 
$\alpha,\beta,\gamma,\delta\in\NN$ with $\alpha,\beta,\gamma,\delta\leq p-1$.
For example, taking $f=x^4+y^4+z^4+w^4$, we see that $X$ is $F$-split if and only if $\chara k\geq 5$.

A similar analysis can be carried out for any cyclic cover of a toric variety.
\end{ex}

\subsection{Affine Quotients}\label{sec:affine}
\begin{defn}
We say that a pair $(Y,\Delta)$ is a \emph{toroidal} if the formal completion of $(Y,\Delta)$ at every closed point $y$ is isomorphic to the formal completion of a pair $(V_y,B_y)$, where $V_y$ is toric and $B_y$ is the toric boundary divisor. 
\end{defn}

\begin{thm}\label{thm:logsmooth}
Let $X$ be an $H$-variety where $H$ has no $p$-torsion, $(Y,\Delta)$ as in \S \ref{ss:setup}, and $(\Ysep,\Deltasep)$ any separation. Assume that $\Ysep$ is affine and $(\Ysep,\lceil \Deltasep \rceil)$ is toroidal. Then $X$ is $F$-regular.
\end{thm}
\begin{proof}
Combine Theorem \ref{thm:main}, Theorem \ref{thm:sep}, and Lemma \ref{lemma:logsmooth} below.
\end{proof}

\begin{lemma}\label{lemma:logsmooth}
Let $Y$ be a normal affine variety, $\Delta$ an effective $\QQ$-divisor, and assume that the coefficients $\Delta$ are all less than $1$. If $(Y,\lceil \Delta\rceil )$ is toroidal, then the pair $(Y,\Delta)$ is $F$-regular. 
\end{lemma}
\begin{proof}
Since $Y$ is affine, we can argue as in the proof of \cite[Proposition 1.1.6]{brion:05a} to show that $(Y,\Delta)$ is $F$-regular if the pair is $F$-regular in a formal neighborhood of each closed point. 
Hence, we are are reduced to showing the following: let $Y$ be toric, $B$ the toric boundary divisor, and $\Delta$ an effective $\QQ$-divisor with $\Delta <B$ and whose coefficients are all smaller than $1$. Then $(Y,\Delta)$ is $F$-regular.

Due to the assumption on the coefficients of $\Delta$, there exists some $e\in \NN$ such that 
$$\lceil(p^e-1)\Delta \rceil+B \leq (p^e-1)B.$$
Now, the canonical toric splitting $\chi^u\to \chi^{u/{p^e}}$ (with $\chi^{u/{p^e}}=0$ if $u/{p^e}\notin M$) splits $\CO_Y\to F_*^e \CO_Y( (p^e-1)B)$, hence also $\CO_Y\to F_*^e \CO_Y( \lceil (p^e-1)\Delta \rceil +B)$, and by Theorem \ref{thm:schwede} we conclude that $(Y,\Delta)$ is $F$-regular.
\end{proof}

\subsection{\texorpdfstring{$\GG_m$-Actions}{Gm-actions}}
The $F$-splitting and $F$-regularity of normal affine $\GG_m$-varieties $X$ with $\CO_X^{\GG_m}=k$ was studied in \cite{watanabe:91a}.
By a classical result of Demazure \cite{demazure:88a}, such $X$ may be described as 
$$X=\spec_Y \bigoplus_{n\in\ZZ_{\geq 0}} \CO(\lfloor nD \rfloor)$$
where 
$$D=\sum_i \frac{p_i}{q_i} P_i$$
 is a $\QQ$-Cartier divisor on a projective variety $Y$. Assuming that $p_i,q_i$ are relatively prime, the \emph{fractional part} of $D$ is 
$$D'=\sum_i \frac{q_i-1}{q_i} P_i.$$
\begin{thm}[{\cite[Theorem 3.3]{watanabe:91a}}]\label{thm:watanabe}
Let $X$ be as above. Then $X$ is $F$-split ($F$-regular) if and only if $(Y,D')$ is $F$-split ($F$-regular).
\end{thm}
Theorem \ref{thm:watanabe} is a special case of our Theorem \ref{thm:main}. Indeed, for $H=\GG_m$ and $X$ as above, $X^\circ/H=Y$  and our $\Delta$ is exactly the $D'$ from above.

\begin{rem}
Suppose now that in the above setting, $Y$ is a complete intersection in $\PP^n$, cut out by hypersurfaces $Y_i$. Assume furthermore that the fractional part of $D$ is of the form
$$
D'=\frac{1-a}{a}(V\cap Y)
$$
for some reduced hypersurface $V\subset \PP^n$. Hara \cite[Theorem 4.2]{hara:95a} shows that 
\[ \spec_Y \bigoplus_{n\in\ZZ_{\geq 0}} \CO\left(\lfloor nD\rfloor\right)\]
is $F$-split if and only if 
\[ \spec_{\PP^n} \bigoplus_{n\in\ZZ_{\geq 0}} \CO\left(\lfloor nE\rfloor\right)\]
is $F$-split for some (or equivalently, for all) ample divisor(s) $E$ on $\PP^n$ with fractional part
$$
E'=\frac{1-a}{a}V+\frac{p-1}{p}\sum Y_i.
$$
Reinterpreted using our notation here, this shows that
$(Y,D')$ is $F$-split if and only if $(\PP^n,E')$ is $F$-split.
\end{rem}
\subsection{Complexity-One Actions}
Let $X$ be a $T$-variety of complexity one, that is, $X$ is a normal variety with an effective action by an algebraic torus $T$ satisfying 
$\dim X=\dim T+1$. Using notation as in \S \ref{ss:setup}, we have that $Y$ is a potentially non-separated smooth curve. Then there is a unique smooth quasiprojective curve $C$ which is a separation of $Y$. Let $\psi:X^\circ \to C$ be the composition of the quotient map $\pi$ with the separation $Y\to C$. For any point $c\in C$, let $\mu(c)$ be the maximal order of the stabilizer of a general point of $\psi^{-1}(c)$.

We can completely characterize $F$-split and $F$-regular complexity one $T$-varieties in terms of the curve $C$ and the stabilizers of the fibers of $\psi$:
\begin{thm}\label{thm:compone}
The complexity-one $T$-variety $X$ is $F$-split in exactly the following cases:
\begin{enumerate}
\item $C$ is affine.\label{case:uno}
\item $C$ is an ordinary elliptic curve, and $T$ acts freely on $X^\circ$.
\item $C=\PP^1$, $\mu(c)=1$ for all but at most three points $c_1,c_2,c_3$, and $\mu(c_1),\mu(c_2),\mu(c_3)$ is one of the triples in Table \ref{table:triples}.\label{case:tres}
\item $C=\PP^1$, $\mu(c)=1$ for all but four points $c_1,c_2,c_3,c_4$ which have  $\mu(c_i)=2$, $p\geq 3$, and the pair $(\PP^1,\sum \frac{1}{2}c_i)$ is ordinary (see Definition \ref{defn:ordinary}).\label{case:quattro}
\end{enumerate} 
Furthermore, $X$ is $F$-regular exactly in case \ref{case:uno}, or case \ref{case:tres} as described in Table \ref{table:triples}.
\end{thm}

\begin{table}
\begin{tabular}{ll@{\qquad}l}
\toprule
triple & $F$-split & $F$-regular\\
\midrule
$(1,*,*)$&$p\geq 2$&$p\geq 2$\\
$(2,2,*)$&$p\geq 3$&$p\geq 3$\\
$(2,3,3)$&$p\geq 5$&$p\geq 5$\\
$(2,3,4)$&$p\geq 5$&$p\geq 5$\\
$(2,3,5)$&$p\geq 7$&$p\geq 7$\\
$(2,3,6)$&$p\equiv 1 \pmod 3$&No\\
$(2,4,4)$&$p\equiv 1 \pmod 4$&No\\
$(3,3,3)$&$p \equiv 1 \pmod 3$&No\\
\bottomrule
\\
\end{tabular}\caption{Stabilizer orders for $F$-split and $F$-regular complexity-one $T$-varieties}\label{table:triples}
\end{table}

\begin{rem}
In the case that $\dim X=2$, $X$ is affine, and $Y$ is projective, we recover \cite[Theorem 4.2]{watanabe:91a}.
\end{rem}

\begin{proof}[Proof of Theorem \ref{thm:compone}]
Consider the separation $(C,\Deltasep)$ of $(Y,\Delta)$. Then by Theorem \ref{thm:sep}, $X$ is $F$-split or $F$-regular if and only if $(C,\Deltasep)$ has the same property.
Suppose first that $C$ is affine. The separation $(C,\Deltasep)$ of $(Y,\Delta)$ is toroidal, so $X$ is $F$-split and $F$-regular by Theorem \ref{thm:logsmooth}. 

For the other cases, we may appeal to \cite[Theorem 4.2]{watanabe:91a} coupled with our Theorem \ref{thm:main}.  However, since the proof of loc.~cit.~is  rather terse, we include a proof here for completeness.
We now consider the case of $C$ projective.
Using the duality of Remark \ref{rem:duality}, we see that a necessary condition for $(C,\Deltasep)$ to be $F$-split (or $F$-regular) is that 
\begin{align}
\label{eqn:deg1}\deg \lceil (p-1)\Deltasep\rceil\leq (p-1)(2g-2)&\qquad\textrm{if $F$-split}\\
\label{eqn:deg2}\deg \lceil (p^e-1)\Deltasep\rceil< (p^e-1)(2g-2)&\qquad\textrm{for some $e$ if $F$-regular}.
\end{align}
Here, $g$ is the genus of the curve $C$. 

Since $\Deltasep$ is effective, we immediate conclude that
$g\leq 1$ if $X$ is $F$-split, with $\Deltasep=0$ in the case $g=1$. In the case $g=1$, we thus conclude that $X$ is $F$-split if and only if $\Deltasep=0$ and $C$ is $F$-split. The condition on $\Deltasep$ is equivalent to saying that $T$ acts freely on $X$, and an elliptic curve is $F$-split if and only if it is ordinary, see e.g. \cite[1.3.9]{brion:05a}. By the above degree requirement, we also see that if $X$ is $F$-regular, then we must have $C=\PP^1$. 

We now analyze the case $C=\PP^1$. 
Let $S$ be the finite subset of $\PP^1$ containing those points $y$ with $\mu(y)\neq 1$.
Note that we have
$$\Deltasep=\sum_{y\in S} \frac{\mu(y)-1}{\mu(y)}y.$$
Assuming that $X$ is $F$-split, the above degree bound leads to
$$
\sum_{y\in S} \frac{1}{\mu(y)} \geq \#S-2,
$$
with strict inequality if $X$ is $F$-regular.
A straightforward calculation shows that the only possible multiplicities $\mu(y)$ which can occur for $X$ $F$-split are the triples listed in Table \ref{table:triples} or $(2,2,2,2)$. Equation \eqref{eqn:deg1} shows that the stated conditions on $p$ are also necessary. Likewise, by \eqref{eqn:deg2} the only multiplicities $\mu(y)$ which can occur for $X$ $F$-regular are the triples listed in Table \ref{table:triples}.
Note that the case $(2,2,2,2)$ is covered in Example \ref{ex:elliptic}, and the case $(3,3,3)$ is covered in Example \ref{ex:trielliptic}.

It remains to show that for each triple, the condition on $p$ is also sufficient for $F$-splitting (or $F$-regularity). Fix the anticanonical divisor $-K_{\PP^1}=\{0\}+\{\infty\}$ as in Example \ref{ex:toricsplit}.
A section 
$$\sum_{i=1-p}^{p-1} a_i \chi^i\in H^0(\PP^1,\CO((1-p)K_{\PP^1}))$$
has a multiple which splits $\CO_{\PP^1}\to F_*\CO_{\PP^1}$ if and only if $a_0\neq 0$.
Now, we may assume that the points $c_1,c_2,c_3$ are respectively $0$, $\infty$, and $1$. 
Let $(\mu(c_1),\mu(c_2),\mu(c_3))$ be one of the triples from Table \ref{table:triples} with $p$ satisfying the requisite bound. Set $$\alpha_i=\frac{\mu(c_i)-1}{\mu(c_i)}.$$ Then $\Deltasep=\sum \alpha_i c_i$, and $\sum \lceil (p-1)\alpha_i \rceil\leq 2(p-1)$.
Hence, there exists $\beta\in \NN$ such that
$$
(\chi^1-1)^{\lceil(p-1)\alpha_3\rceil}\cdot \chi^{-\beta}
$$
is a section of $H^0(\PP^1,F_*\CO(\lceil (1-p)(K_{\PP^1}+\Deltasep)\rceil)$ and the coefficient $a_0$ of $\chi^0$ is non-zero. We conclude that a multiple of this section corresponds to a splitting $\sigma$ of $(\PP^1,\Deltasep)$, so $X$ is $F$-split.

If we are in the situation where we are claiming that $X$ is $F$-regular, then there exists $e\in \NN$ such that $\sum \lceil (p^e-1)\alpha_i \rceil < 2(p^e-1)$. Composing the splitting $\sigma$ from above with itself, we get that $\sigma^e$ splits $\CO_{\PP^1}\to F^e_*\CO_{\PP^1}(\lceil (p^e-1)\Deltasep\rceil)$; Let $\tau$ be the corresponding section of  $H^0(\PP^1,F_*^e  \CO_{\PP^1}(\lceil (1-p^e)(K_{\PP^1}+\Deltasep)\rceil))$. 
 By choice of $e$, we have
$$\lceil (1-p^e)(K_{\PP^1}+\Deltasep)\rceil>0,$$
so there exists an effective divisor $D>0$ with $\tau$ a section of $F_*^e  \CO_{\PP^1}(\lceil (1-p^e)(K_{\PP^1}+\Deltasep)\rceil-D)$.
Hence, $\sigma^e$ splits $\CO_{\PP^1}\to F_*^e\CO_{\PP^1}(\lceil (p^e-1)\Deltasep\rceil-D)$, so by Theorem \ref{thm:schwede}, $(\PP^1,\Deltasep)$ and $X$ are $F$-regular.
\end{proof}

\begin{rem}\label{rem:genus}
We can define the genus of a pair $(C,\Deltasep)$ by 
  \[g(C,\Deltasep)=\frac{\deg\Deltasep + 2g(C)}{2}.\]
By the above theorem, $F$-split implies that $g(C,\Deltasep) \leq 1$, and $F$-regular that $g(C,\Deltasep) < 1$. 
\end{rem}

\begin{ex}
By \cite[Proposition 6.3]{smith:00a}, any smooth Fano variety in characteristic zero is $F$-regular after reducing to characteristic $p$ for $p$ sufficiently large. We illustrate this with the list of complexity-one smooth Fano threefolds from \cite{suess:14a}. For the threefolds 2.24, 3.8, and 3.10, the stabilizer orders are given by the triple $(2,2,2)$.\footnote{Note that only special elements in these three deformation families admit a two-torus action.} Hence, these threefolds are $F$-split and $F$-regular exactly in characteristics $p$ with $p\geq 3$. All other threefolds on the list have stabilizer orders given by the triple $(1,*,*)$ and are $F$-split and $F$-regular in arbitrary characteristic.
\end{ex}

\subsection{Surjectively Graded Algebras}\label{sec:surj}
Let $A$ be an $F$-finite noetherian $\ZZ^n$-graded normal integral domain of characteristic $p$. Then $A$ is \emph{surjectively graded} \cite{hashimoto:03a} if for all $u,u'\in\ZZ^n$ with $A_u,A_{u'}\neq 0$, the multiplication map $A_u\otimes_{A_0}A_{u'}\to A_{u+u'}$ is surjective.
Then Hashimoto shows the following:
\begin{thm}[{cf. \cite[Theorem 5.1]{hashimoto:03a}}]
Assume that $\bigoplus_{n\in \ZZ} A_{nu}$ is $F$-regular for some $u$ in the interior of the weight cone of $A$ with $A_u\neq 0$. Then $A$ is $F$-regular as well. 
\end{thm}

Surjectively graded algebras fit nicely into our framework as well. Let $A$ be a surjectively graded finitely generated normal $k$-algebra, $X=\spec A$. 
Constructing $X^\circ$, $Y$, and $\Delta$ as in \S\ref{ss:setup}, we have that $\Delta=0$. Indeed, if $A$ is surjectively graded, then the sheaf
\[\mcA\cong \bigoplus_{v\in M} \CO_Y(\lfloor \D(v)\rfloor) \]
of $\CO_Y$-algebras is also (locally) surjectively graded. 
But it is straightforward to check that this implies that $\D(v)$ is integral for all $v\in M$, and hence $\Delta=0$.
We may thus conclude by our Theorem \ref{thm:main} that $X=\spec A$ is $F$-regular if any only if $X^\circ/T$ is $F$-regular. 

\subsection{Cox Rings and Related Constructions}\label{sec:cox}
Let $Y$ be a normal variety with finitely generated class group $\Cl(Y)$. The \emph{Cox sheaf} of $Y$ is the $\Cl(Y)$-graded sheaf
\[\R(Y)=\bigoplus_{[D]\in\Cl(Y)}\CO_Y(D).\]
This definition appears to depend on choice of representatives $D$ of classes $[D]\in\Cl(Y)$, but any two choices lead to isomorphic sheaves. Furthermore, choosing representatives for a generating set of $\Cl(Y)$ leads to an $\CO_Y$-algebra structure on $\CO_Y$, and any two choices lead to isomorphic $\CO_Y$-algebras \cite[\S 1.4]{coxbook}. 

The \emph{Cox ring} of $Y$ is the ring $R(Y)=H^0(Y,\R(Y))$. It is a natural generalization of the homogeneous coordinate ring of projective space. Note that in general it need not be finitely generated. However, it is always integral and normal \cite[\S 1.5]{coxbook}.
\begin{prop}[{\cite[cf. Proposition 4.6]{gongyo:12a}}]\label{prop:cox}
Suppose that $R(Y)$ is finitely generated, and assume that $\Cl(Y)$ has no $p$-torsion. Then 
\[X=\spec R(Y)\]
is $F$-split ($F$-regular) if and only if $Y$ is $F$-split ($F$-regular).
\end{prop}
\begin{proof}
Let $H=\spec k[\Cl(Y)]$. 
In this situation, $U=\spec_{\CO_Y} \R(Y)$ is an $H$-invariant open subset of $X$ of codimension at least one \cite[\S 1.6]{coxbook}, hence $X$ is $F$-split ($F$-regular) exactly when $U$ is. But by construction, $H$ acts on $U$ with finite stabilizers, $U/H=Y$, and the boundary divisor $\Delta\subset Y$ is trivial. The claim now follows from our Theorem \ref{thm:main}. 
\end{proof}

A related situation occurs when considering a normal variety $Z$ embedded in some other variety $Y$ as above with $R(Y)$ finitely generated.
For simplicity, we shall assume that $Y$ is toric, in which case $R(Y)$ is a polynomial ring; generalizations are left to the reader.
Consider $Z\subset Y$ a normal variety. Let $U\subset \spec R(Y)$ be as in the proof of Proposition \ref{prop:cox}, and let $I\subset R(Y)$ be the $\Cl(Y)$-homogeneous ideal of $\overline{\pi^{-1}(Y)}\subset \sec R(Y)$, where $\pi:U\to Y$ is the quotient map.
\begin{prop}
Let $Z$, $Y$, and $I$ be as above.
\begin{enumerate}
\item If $V(I)$ is $F$-split ($F$-regular), then so is $Z$.
\item Suppose that $Y$ is smooth, $V(I)$ normal, and any component of $V(I)\setminus U$ of dimension $\dim V(I)-1$ has infinite $H$-stabilizer. Then $Z$ being $F$-split ($F$-regular) implies that $V(I)$ is as well.
\end{enumerate}
\end{prop} 
\begin{proof}
The first claim is a straightforward application of our Theorem \ref{thm:main}. For the second claim, note that $Y$ smooth implies that $H$ acts freely on $U$. Then $Z$ is $F$-split ($F$-regular) if and only if $V(I)\cap U$ is by loc.~cit. Under the further assumptions, $V(I)$ is $F$-split ($F$-regular) if and only if $V(I)\cap U$ is.
\end{proof}

\begin{ex}[Elliptic curves in $\PP^1\times \PP^1$]
Any form $f\in k[x_0,x_1,y_0,y_1]$ of bidegree $(2,2)$ defines a (possibly singular) elliptic curve $E$ embedded in $\PP^1\times \PP^1$. Assume that $V(f)\subset \A^4$ is normal. Then the corresponding curve $E$ is $F$-split if and only if the coefficient of $(x_0x_1y_0y_1)^{p-1}$ in $f^{p-1}$ is non-zero. Indeed, by the above proposition, $E$ is $F$-split if and only if $k[x_0,x_1,y_0,y_1]/(f)$ is $F$-split, and Fedder's criterion \cite{fedder:83a} implies that the latter is $F$-split if and only if $f^{p-1}\notin\langle x_0^p,x_1^p,y_0^p,y_1^p\rangle$.

 \end{ex}

\section{Toric vector bundles}\label{sec:vb}
Toric vector bundles and their projectivizations provide a natural class of normal varieties with action by a lower-dimensional torus. We apply our general results here to discuss the $F$-splitting and $F$-regularity of certain toric vector bundles. Note that in relation to positivity properties of toric vector bundles, it was asked in \cite{hering:10a} exactly which toric vector bundles are $F$-split. Our Theorems \ref{thm:main} and \ref{thm:vb-quotient} give a complete answer in the special case of \emph{two-step} bundles defined below.

Given a vector bundle $\E$, we denote by $\PP(\E)$ the corresponding projective bundle, whose fibers are the spaces of lines in the fibers of $\E$.\footnote{By some authors, this bundle is denoted by $\PP(\E^*)$.}
To begin with, we have the following well known result.
\begin{prop}[{cf. \cite[\S 1.1]{brion:05a}}]
A vector bundle $\E$ is $F$-split ($F$-regular) if and only if its projectivization $\PP(\E)$ is $F$-split ($F$-regular).  
\end{prop}
\begin{proof}
  There is a natural $\GG_m$-action on $\E$ given by the diagonal action on each fiber. The fixed point set is given by the zero-section and $\GG_m$ acts with trivial stabilizers elsewhere. Now, $(\PP(\E),0)$ is the corresponding quotient pair and we obtain the result by applying Theorem \ref{thm:main}.
\end{proof}

Now, let $X$ be a toric variety corresponding to a fan $\Sigma$ with embedded torus $T$, see \cite{fulton:93a} for details. Throughout this section, $M$ will be the character lattice of $T$, and $\Sigma^{(1)}$ the set of rays of $\Sigma$. A \emph{toric vector bundle} on $X$ is a vector bundle $\E$ on $X$ equipped with a $T$-equivariant structure. This equivariant structure turns both $\E$ and $\PP(\E)$ into $T$-varieties.

To a toric vector bundle $\E$ of rank $r$, Klyachko \cite{klyachko:89a} associated a $k$-vector space $E$ of dimension $r$ and a full decreasing filtration $E^\rho(\lambda)$ of $E$ for every ray $\rho \in \Sigma^{(1)}$
\[
\cdots \supset E^{\rho}(\lambda-1) \supset E^{\rho}(\lambda) \supset E^{\rho}(\lambda+1) \supset \cdots.
\]
fulfilling the following compatibility condition:
For each maximal cone $\sigma \in \Sigma$, there are lattice points $u_1, \ldots, u_r \in M$ and a decomposition into one-dimensional subspaces $E = L_1 \oplus \cdots \oplus L_r$ such that 
\[
E^{\rho}(\lambda) = \bigoplus_{\rho(u_i) \, \geq \, \lambda} L_i,
\]
 for each $\rho \preceq \sigma$ and all $\lambda \in \ZZ$. Here $\rho(u_i)$ denotes the value of a primitive generator of $\rho$ on $u_i$.
From this data one can reconstruct $\E$ as follows. The sections of $\E$ on the chart $U_\sigma$ of $X$ corresponding to $\sigma$ are given as a submodule of $k[\sigma\dual \cap M] \otimes E$ via
\[
H^0(U_\sigma,\E)_u = \bigcap_{\rho \preceq \sigma} E^{\rho}(\rho(u)).
\]

Note, that the description of toric vector bundles by filtration behaves well with standard constructions as tensor product and dualization. Indeed, the dual bundle corresponds to the filtrations $E^{*\rho}(\lambda) = E^\rho(-\lambda)^\perp$ of the dual vector space $E^*$.

\begin{defn}
We say that a toric vector bundle $\E$ is a \emph{two-step bundle} if every filtration $E^\rho(\lambda)$ has most two steps where the dimension jumps (i.e.~at most one proper subset of $E$ occurs).
\end{defn}
Clearly, any rank two toric vector bundle is a two-step bundle, since $E$ is two-dimensional in this case. 

\begin{ex}
 By \cite{klyachko:89a} the tangent and cotangent bundles are examples of two-step bundles, since their filtrations have the following form:

 \begin{equation*}                             
\mathcal{T}^{\rho}(\lambda) = \left \{ \begin{array}{ll} N \otimes k & \text{ for } \lambda < 0 \\
								  \langle \rho \rangle   & \text{ for } \lambda = 0, \\
								  0 & \text{ for } \lambda > 0,
				  \end{array} \right. 
\end{equation*}
 \begin{equation*}                             
\Omega^{\rho}(\lambda) = \left \{ \begin{array}{ll} M \otimes k & \text{ for } \lambda < 0 \\
								  \rho^\perp  & \text{ for } \lambda = 0, \\
								  0 & \text{ for } \lambda > 0.
				  \end{array} \right.
\end{equation*}
\end{ex}

For a given two-step bundle $\E$, let $E_1,\ldots,E_\ell$ be the proper subspaces of $E$ occurring in the filtrations $E^\rho(\lambda)$. For every $E_i$ and every ray $\rho \in \Sigma^{(1)}$ we define 
\[\mu_i(\rho) = \max \{\lambda \mid  E_i \subset E^\rho(\lambda)\} -\min \{ \lambda \mid E^\rho(\lambda) \neq E\}.\]
Now we set 
\[\mu_i = \max\{\mu_i(\rho) \mid \rho \in \Sigma^{(1)}\}.\]

Consider $Y = \Bl_{F_1,\ldots,F_\ell}\PP(E)$, the successive blowup of $\PP(E)$ in the strict transforms of the subspaces $F_i$. Let the corresponding (strict transforms) of the exceptional divisors be denoted by $D_i$; we define the exceptional divisor of the blowup in a hyperplane to be the hyperplane itself. Note that $Y$ and this configuration of divisors is independent of the ordering of the $E_i$ on a big open subset.

\begin{thm}
\label{thm:vb-quotient}
Let $(Y,\Delta)$ be the quotient pair for $\PP(\E)$, where $\E$ is a two-step toric vector bundle. A separation of $(Y,\Delta)$ is given by
\[\Ysep = \Bl_{E_1,\ldots,E_\ell} \PP(E),\qquad  \Deltasep= \sum_i \frac{\mu_i-1}{\mu_i} D_i. \]
\end{thm}
\begin{proof}
The claim follows directly from the arguments of \cite[Proposition 3.5 and \S 6.2]{gonzalez:12a} and \cite[Theorem 5.9]{hausen:10a}.
\end{proof}

For two-step toric vector bundles $\E$, we can thus apply Theorem \ref{thm:sep} together with Theorem \ref{thm:main} to determine when $\PP(\E)$ and $\E$ are $F$-split or $F$-regular.
In the following, we consider several special cases.

Note that the following Corollary was obtained by \cite{xin:14a} for the case of $F$-splitting using arguments different than ours:
\begin{cor}\label{cor:cotangent}
  The cotangent bundle of a smooth toric variety $X$ is always $F$-regular. In particular, it is $F$-split.
\end{cor}
\begin{proof}
  In this case, we have $(\Ysep,\Deltasep)=(\PP^{n-1},0)$, where $n=\dim X$.
\end{proof}

On the other hand, the Frobenius pullback of the cotangent bundle is not even $F$-split:
\begin{ex}
  We consider the vector bundle $\mathcal{E}=F^*\Omega_X$ on a smooth complete toric variety $X=X_\Sigma$ of dimension $n$. This bundle is given by the filtrations 
  \[
  E^{\rho}(\lambda) = 
  \begin{cases}
    M\otimes k     & \lambda < 0 \\
    \rho^\perp & 0 \leq \lambda \leq p\\
    0     & \lambda > p    
  \end{cases}
  \]
In particular it is a two-step bundle and we see that a separation of the corresponding quotient pair is given by $(\Ysep, \Deltasep)$, where $\Ysep=\PP(M\otimes k)$ and 
$$\Deltasep=\sum_{\rho\in\Sigma^{(1)}} \frac{p-1}{p}\PP(\rho^\perp).$$
 We obtain 
  \[\deg \lfloor (1-p)(K_{\Ysep} + \Deltasep)\rfloor = (p-1)(n - (\#\Sigma^{(1)})).\]
Since $\#\Sigma^{(1)}>n$ for $X$ complete, the right hand side is negative, and we conclude using Remark \ref{rem:duality} that $\E$ cannot be $F$-split.
\end{ex}

It is known that the cotangent bundle for flag varieties is also $F$-split, see \cite{kumar:99a}. We ask:
\begin{question}
Let $X$ be any smooth $F$-split (or $F$-regular) variety. Is $\Omega_X$ always $F$-split (or $F$-regular)?
\end{question}

The tangent bundle on a smooth toric variety is not always $F$-split (see Example \ref{ex:tbundle} below), but it is in the case of projective space.
\begin{cor}
  The tangent bundle of $\PP^{n}$ is always $F$-regular. In particular, it is $F$-split.
\end{cor}
\begin{proof}
 In this case, we have $\Ysep$ is the blowup of $\PP^{n-1}$ in $n+1$ general points, and $\Deltasep=0$. The claim now follows from the above discussion and the following lemma.
\end{proof} 
\begin{lemma}
The blowup of $\PP^{n}$ in $n+2$ general points is $F$-regular.
\end{lemma}
\begin{proof}
After applying a projective transformation, we can take the $n+2$ points to be $n+1$ toric fixed points of $\PP^n$, along with the point $1$.
The blowup $X$ of $\PP^n$ in the $n+1$ fixed points is toric, and choosing $-K_X$ to be the standard toric anticanonical divisor, a basis for $H^0(X,\CO(-K_X))$ is given by monomials $\chi^u$ for $u=(u_1,\ldots,u_n)\in\ZZ^n$ satisfying
\begin{align*}
&-1\leq u_i\leq n-1 &&i=1,\ldots,n;\\
&1-n\leq \sum u_i\leq 1. &&
\end{align*}
Consider the global section
$$
\tau=(1-\chi^{e_n})\prod_{i=1}^{n-1}(1-\chi^{-e_i})
$$
Here $e_i$ is the standard basis of $\ZZ^n$. The coefficient of $\chi^0$ in $\tau^{p-1}$ is $1$. Hence, under the Grothendieck duality used in Remark \ref{rem:duality}, $\tau^{p-1}$ corresponds to an $F$-splitting of $X$. Furthermore, this lifts to an $F$-splitting of the blowup $\widetilde{X}$ of $X$ in the point $1$, since $\tau$ vanishes to order $n$ at the point $1$, see \cite[Exercise 1.3.13]{brion:05a}.

But in fact, $\tau$ is a global section of $\CO(-K_X-E)$ for $E$ any one of the exceptional divisors of $X\to \PP^n$, excluding one. Hence, by  Grothendieck duality, we have a splitting of $\CO_{\widetilde{X}}\to F_* \CO_{\widetilde{X}}(E)$. But $\widetilde{X}\setminus E$ is an open subvariety of a toric variety, so $\widetilde{X}$ is $F$-regular by Theorem \ref{thm:schwede}.
\end{proof}

We can give the most precise answer as to when $\E$ is $F$-split or $F$-regular in the case of rank two toric vector bundles.
\begin{cor}\label{cor:rktwo}
Let $\E$ be a rank two toric vector bundle with associated vector space $E$ and proper lines $E_i$. Then $\E$ is $F$-split if and only if either there are at most three lines $E_i$ with values $\mu_i>1$, and the $\mu_i$ form a triple as in Table \ref{table:triples}; or $p\geq 3$, there are exactly four lines $E_1,E_2,E_3,E_4$ with $\mu_i>1$, for these four lines we have $\mu_i=2$, and the coefficient of $y^{(p-1)/2}$ in $(y-\lambda)^{(p-1)/2)}(y-1)^{(p-1)/2}$ is non-zero, where $\lambda$ is the cross-ratio of four colinear points $v_i\in E_i$, $i=1,2,3,4$.

Likewise, $\E$ is $F$-regular if and only if either there are at most three lines $E_i$ with values $\mu_i>1$, and the $\mu_i$ form a $F$-regular triple as in Table \ref{table:triples}.
\end{cor}
\begin{proof}
This is a direct application of Theorem \ref{thm:compone}.
\end{proof}

N.~Lauritzen raised the question if there is an $F$-split vector bundle $\E$ such that the dual bundle $\E^*$ is not $F$-split \cite{openproblems}. Corollary~\ref{cor:rktwo} implies that for toric bundles of rank two this cannot happen, since in the cases where $\E$ is F-split the quotient pairs of $\E$ and $\E^*$ are isomorphic. In the following examples, we will see a number of two-step toric bundles of higher rank satisfying this property.
\begin{ex}\label{ex:lauritzen}
Let $X=X_\Sigma$ be a toric variety, $\Sigma^{(1)} = \{\rho_1, \ldots, \rho_\ell\}$. We consider the toric vector bundle $\E$ on $X$ given by
 \[
  E^{\rho_i}(\lambda) = 
  \begin{cases}
    E     & \lambda < 0 \\
    E_i & \lambda = 0\\
    0     & \lambda > 0    
  \end{cases}
  \]  
where the $E_i$ are  hyperplanes in $E$ in sufficiently general position, and $\dim E=n+1$.
If for example $X$ is regular, this collection of filtrations fulfills the necessary compatibility condition.

Now, a separation for the quotient pair of $\PP(\E)$ is given by $(\PP^n,0)$, and for $\PP(\E^*)$ by 
$(\Bl_{\ell} \PP^n, 0)$, where $\Bl_{\ell} \PP^n$ is the blowup of $\PP^n$ is $\ell$ general points.
Since $\PP^n$ is $F$-regular, $\PP(\E)$ is as well; in particular, it is $F$-split. On the other hand, on $\Ysep=\Bl_{\ell} \PP^n$ the sheaf $\CO((1-p)K_\Ysep)$ has no global section if $\ell \geq h^0(\PP^n,\CO(n+1))$, so in this case $\PP(\E^*)$ cannot be $F$-split.

We can modify this example to give a counterexample where $\PP(\E)$ is $F$-split but not $F$-regular, and $\PP(\E^*)$ is not $F$-split. Indeed, 
consider the bundle $\E$ as above, except that for $2(n+1)$ of the rays the filtrations $E^{\rho_i}(\lambda)$ have value $E_i$ for two steps in the filtration instead of just one. In this case, the boundary divisor $\Delta$ is $\sum_i \frac{1}{2}E_i$, the sum being over the indices for those $2(n+1)$ rays.
Since $\deg \Delta=n+1$, $(Y,\Delta)$ cannot be $F$-regular, but it will be $F$-split if the $E_i$ are sufficiently general. On the other hand, $\PP(\E^*)$ will still not be $F$-split.
\end{ex}

\begin{ex}\label{ex:tbundle}
In \cite[Example 4.2]{gonzalez:12a}, a smooth toric variety is constructed such that the quotient for the tangent bundle is given by $\Ysep=\Bl_{14}\PP^{n-1}$. In characteristic $p \neq 2,3$, $9$ of the $14$ points form the complete base locus of a pencil of cubics \cite{totaro:08a}. Hence, $\omega_\Ysep^{1-p}$ does not admit any global sections and $\Ysep$ and hence $\mathcal{T}_X$ is not $F$-split. On the other hand, $\Omega_X$ is always $F$-split by Corollary~\ref{cor:cotangent}.
\end{ex}

The situation for toric rank two bundles motivates the following modified version of Lauritzen's question.

\begin{question}
   Is there an $F$-split (non-toric) rank two vector bundle $\E$ such that the dual bundle $\E^*$ is not F-split? 
\end{question}

\section{Diagonal splittings}\label{sec:diag}
\begin{defn} 
A \emph{diagonal splitting} of a scheme $X$ is a splitting of $X\times X$  compatible with the diagonal \cite{ramanathan:87a}. 
By a \emph{diagonal splitting of a pair} $(X,\Delta)$ we mean a 
splitting of \[(X \times X,\;\Delta\!\times\!X + X\!\times\!\Delta )\] which is compatible with the diagonal. 
More generally, by a \emph{diagonal splitting of a triple} $(X;\Delta_+,\Delta_-)$  we mean a 
splitting of \[(X \times X,\;\Delta_+\!\times\!X + X\!\times\!\Delta_- )\] which is compatible with the diagonal. 
\end{defn}
\noindent Note that $X$ being diagonally split has strong consequences for the syzygies of $X$, see e.g. \cite[1.5]{brion:05a}

\begin{ex}
\label{ex:diag-splitting-curves}
If $C$ is a complete curve and $(C,\Delta)$ is diagonally split, then $g(C,\Delta)\leq\nicefrac{1}{2}$ has to hold (see Remark \ref{rem:genus} for a definition of $g(C,\Delta)$). Likewise, if $(C;\Delta_+, \Delta_-)$  is diagonally split, then we must have $g_+:=g(C,\Delta_+) \leq \nicefrac{1}{2}$ and $g_- := g(C,\Delta_-) \leq \nicefrac{1}{2}$.
\end{ex}
\begin{proof}
 The diagonal has bidegree $(1,1)$ in $C \times C$. Hence,
\[D:=-K_{C\times C}-(\Delta_+\!\times\!X + X\!\times\!\Delta_-)\]
 has bidegree  $(2-2g_+,\;2-2g_-)$. Hence, if $g_+$ or $g_-$ is larger than $\nicefrac{1}{2}$, then $D$ and all its positive multiples have empty linear systems. By Remark~\ref{rem:duality} this implies that there is no such splitting. For $(C,\Delta)$ we get the claim by considering $(C; \Delta,\Delta)$.
\end{proof}

Now, let $X$ be an $H$-variety as in \S \ref{ss:setup} and assume that $H$ has no $p$-torsion. The product $X\times X$ admits a natural $H\times H$-action. However, the diagonal is invariant only with respect to the diagonal subgroup $H \subset H\times H$. This embedding of groups corresponds to the surjection of character lattices 
\[M\times M \to M;\quad (u_1,u_2) \mapsto u_1+u_2.\]
Hence, semi-invariant functions of degree $(u,-u)$ with respect  to the $H\times H$-action are exactly the invariant functions with respect to the diagonal action. Now, by using Proposition~\ref{prop:comp2}, we see that we may assume that a diagonal splitting of $X$ is of the form
\begin{equation}\label{eqn:diag}
\sigma =  \sum_{w \in M} \sigma_{(w,-w)}
\end{equation}
where $\sigma_{(w,-w)} \in \Hom_{\CO_X}(F_*\CO_{X\times X},\CO_{X\times X})_{(w,-w)}$.

\begin{rem}
\label{rem:pushforward-to-quotient}
Note that by \S\ref{sec:tfrob}, given an element $\sigma_w \in \Hom_{\CO_X}(F_*\CO_X,\CO_X)_{w}$ we may interpret is as an element of $\Hom_{\CO_Y}(F_*k(Y),k(Y))$ which we as before we will denote by $\bar \sigma_w$ in the following. Remember that Lemma~\ref{lemma:piotr} ensures that \[\bar \sigma_w \in \Hom_{\CO_Y}(F_*\CO_Y(\lceil (p-1)\Delta + \D(w) \rceil),\CO_Y).\]

This extends to $X \times X$ with the full $H\!\times\!H$-action as follows. For \[\sigma_{(w,-w)}=\sum_{i} \sigma_i \otimes \sigma^i \in \Hom_{\CO_X}(F_*\CO_X,\CO_X)_{w} \otimes \Hom_{\CO_X}(F_*\CO_X,\CO_X)_{-w}\] 
we have $\bar \sigma_{(w,-w)} = \sum \bar \sigma_i \otimes \bar \sigma^i$. This defines an element of $\Hom_{\CO_V}(F_*\CO_V,\CO_V)$ of some open subset of $V \subset Y \times Y$ intersecting the diagonal.
\end{rem}

We now give a characterization of those invariant splittings $\sigma$ of $X\times X$ which are compatible with the diagonal. For simplicity, we will assume that $H$ is equal to a torus $T$.
For every class $[w] \in M/pM$ we define \[\bar \sigma_{[w]} = \sum_{u \in [w]} \bar{\sigma}_{(u,-u)}.\]
In the following we denote the ideal sheaves of the diagonals in $X \times X$ and $Y \times Y$ by $I_X$ and $I_Y$, respectively.
\begin{thm}
\label{thm:diag-criterion}
  A Frobenius splitting $\sigma$ of $X \times X$ is compatible with the diagonal if and only if for every $[w] \in M/pM$ we have $\bar{\sigma}_{[w]} \equiv \bar{\sigma}_{[0]}\; (\mathsf{mod}\; I_Y)$  and $\bar{\sigma}_{[w]}$ is compatible with the diagonal, that is, $\bar{\sigma}_{[w]}(F_* I_Y)\subset I_Y$.
\end{thm}
\begin{proof}
 We consider generators of $F_* I_X$  as an $\CO_{X\!\times\!X}$-module. There are two types of generators we have to take into account. One coming from the diagonal of $T \times T$ the other one from the diagonal of $Y \times Y$:
\begin{align}
  f \cdot \chi^0 \otimes \chi^{0} - f \cdot \chi^u \otimes \chi^{-u}&, \qquad u \in M; f \in k(Y) \times k(Y)
\label{eq:ideal-1}
\\
  f \;\cdot\; (\chi^w \otimes \chi^u)&, \qquad u,w \in M \text{ and } f \in I_Y.\label{eq:ideal-2}
\end{align}
In fact, these elements generate $F_* I_X$ as a $k$-vector space.

Assume first that we have an element $g$ of the form (\ref{eq:ideal-2}). Then $\sigma(g)$ will vanish if $u \neq -w$. Assume that  $u = -w$. We obtain
  \begin{align*}
    \sigma(f \;\cdot\; (\chi^w \otimes \chi^{-w})) = &  \sum_{u \in [w]} \sigma_{(u,-u)}(f \;\cdot\; (\chi^w \otimes \chi^{-w}))\\
    = &  \sum_{u \in [w]} \bar{\sigma}_{(u,-u)}(f) \;\cdot\; \chi^{w-u} \otimes \chi^{u-w}\\
    = &  \left(\sum_{u \in [w]} \bar{\sigma}_{(u,-u)}(f)\right)\chi^0 \otimes \chi^0 \quad + \\
   & + \quad  \sum_{u \in [w]}   \bar{\sigma}_{(u,-u)}(f)\cdot(\chi^{w-u} \otimes \chi^{u-w} - \chi^0 \otimes \chi^0).
  \end{align*}
Note that the first summand of the right-hand-side is an element of $I_X$ if and only if  $\sum_{u \in [w]} \bar{\sigma}_{(u,-u)}(f)=\bar \sigma_{[w]}(f)$ is an element of $I_Y$. The second summand is always an element of $I_X$, since $(\chi^{w-u} \otimes \chi^{u-w} - \chi^0 \otimes \chi^0)$ lies $I_X$.

Assume instead we have an element $g$ of the form (\ref{eq:ideal-1}). Then we obtain
  \begin{align*}
     \sigma(g) &= \sigma(f \cdot \chi^0 \otimes \chi^{0} - f\cdot\chi^u \otimes \chi^{-u})\\
 &=  \sigma(f \cdot \chi^0 \otimes \chi^{0}) - \sigma(f\cdot\chi^u \otimes \chi^{-u})\\
&= \sum_{u \in [0]}\sigma_{(u,-u)}(f \cdot \chi^0 \otimes \chi^{0}) -  \sum_{u \in [w]}\sigma_{(u,-u)}(f \cdot \chi^w \otimes \chi^{-w})\\
&= \sum_{u \in [0]}\bar \sigma_{(u,-u)}(f) \cdot \chi^{-u} \otimes \chi^{u} -  \sum_{u \in [w]}\bar \sigma_{(u,-u)}(f) \cdot \chi^{w-u} \otimes \chi^{u-w}\\
&\equiv (\bar \sigma_{[0]}(f) - \bar \sigma_{[w]}(f)) \cdot  \chi^0 \otimes \chi^0.
  \end{align*}
Here, the congruence is modulo elements of the form $( \chi^0 \otimes \chi^0 - \chi^{u} \otimes \chi^{-u}) \in I_X$ as above.
Now, the right-hand-side lies in $I_X$ if and only if $(\bar \sigma_{[0]}(f) - \bar \sigma_{[w]}(f))$ is an element of $I_Y$.
\end{proof}

We obtain the following corollary,  which is a simple generalization of the corresponding result on toric varieties in \cite{payne:09a}.
\begin{cor}
\label{prop:section-weight-criterion}
  Consider a $T$-variety $X$, and suppose $\sigma$ is a splitting of $X\times X$ compatible with the diagonal. Then for every class $[w] \in M/p M$ there must be a representative $u \in [w]$ such that the homogeneous component of weight $(u,-u)$ in $\sigma$ is non-trivial. In particular $\Hom_{\CO_X}(F_*\CO_X, \CO_X)_u \neq 0$ and $\Hom_{\CO_X}(F_*\CO_X, \CO_X)_{-u} \neq 0$. 
\end{cor}
\begin{proof} For $\sigma$ to be a splitting, $\bar \sigma_{[0]}$ must be non-trivial, and the result follows by Theorem \ref{thm:diag-criterion}.
\end{proof}
\begin{rem}[The toric case]
In the toric case, the criterion that for all $[w] \in M/p M$, there must be a representative $u \in [w]$ such that $\Hom_{\CO_X}(F^*\CO_X, \CO_X)_{\pm u} \neq 0$ is exactly the criterion that the polytope $\FF_X:=P_X\cap -P_X$ contains a representative of every class $[w] \in M/p M$, cf. Lemma \ref{lemma:extends}. Payne shows that this criterion is both necessary and sufficient \cite{payne:09a}. The sufficiency of this criterion is easily seen: for any lattice point $u\in \FF_X$, $\Hom_{\CO_X}(F^*\CO_X, \CO_X)_{\pm u} \cong k$ by Remark \ref{rem:bijection}. Since $X$ is complete, $[0]\cap 
(P_X\cap -P_X)=0$, so by Lemma \ref{lemma:invariant}, $\sum_u 1\cdot \chi^u\otimes \chi^{-u}$ corresponds to an invariant splitting $\sigma$ of $X\times X$, where the sum is taken over a choice of representative $u$ for each class of $M/pM$. Now, by Theorem \ref{thm:diag-criterion}, this splitting is compatible with the diagonal.

It was Payne's result which was one of our original motivations for studying $F$-splittings of higher complexity $T$-varieties. As Payne points out, the diagonal of $X\times X$ is not $T\times T$-invariant, but it is invariant with respect to the action of the diagonal torus. We were struck by the fact that Payne's polytope $\mathbb{F}_X=P_X\cap -P_X$ is exactly the polytope corresponding to the anticanonical divisor on the Chow quotient $Z$ of $X\times X$ by the diagonal torus $T$. In fact, our machinery (\S \ref{sec:tfrob} and Proposition \ref{prop:comp2}) can be used to show that a toric variety $X$ is $F$-split if and only if the above quotient $Z$ is split compatibly with some point in the interior of $Z$ (note that $Z$ is a toric variety with respect to the quotient torus $(T\times T)/T$). This is easily seen to be equivalent to Payne's criterion discussed above. We leave the details to the reader. 
\end{rem}

Our next goal is to give a simpler necessary condition for a complexity-one $T$-variety to be diagonally split. To begin with, suppose that $Y$ is any complete variety, and let $\D:M\to \Div_{\QQ}(Y)$ be as in \eqref{eqn:d}. We set
\[U=\spec_{Y}\bigoplus_{u \in M} \CO(\D(u)).\]
Then the quotient pair of $U$ is $(Y,\Delta)$.
\begin{lemma}
\label{lem:shifting}
Assume we are given a diagonal splitting  $\sigma$ of $U$ of the  form \eqref{eqn:diag}. Let $\Delta_+,\Delta_-$, be effective $\QQ$-divisors on $Y$. Suppose that for every $w\in pM$ with $\sigma_{(w,-w)}$ non-trivial there are functions $f^p_w \in K(Y) \subset F_*K(Y)$  satisfying $f_0^p = 1$ and with
\[\div(f^{\pm p}_{w}) +  \lceil (p-1)\Delta + \D(\pm w)\rceil \geq \lceil(p-1)\Delta_\pm\rceil.\] 
Then there
is a diagonal splitting of $(Y;\Delta_+,\Delta_-)$.
\end{lemma}
\begin{proof}
Remember that using Remark~\ref{rem:pushforward-to-quotient} we obtain $\bar \sigma_{(w,-w)}$ as an element of 
\[\Hom_{\CO_Y}(F_*\CO_Y(\lceil(p-1)\Delta + \D(w)\rceil),\CO_Y) \otimes \Hom_{\CO_Y}(F_*\CO_Y(\lceil(p-1)\Delta + \D(-w)\rceil),\CO_Y).\]
By our hypothesis on the $f^p_{w}$, multiplying with $f^p_{w} \otimes f^{-p}_{w}$ gives an element of
\[\Hom_{\CO_Y}(F_*\CO_Y(\lceil(p-1)\Delta_+\rceil,\CO_Y) \otimes \Hom_{\CO_Y}(F_*\CO_Y(\lceil(p-1)\Delta_-\rceil,\CO_Y).\]

We set 
  \[\sigma'  =  \sum_{w} (f^p_{w}\chi^{w} \otimes f^{-p}_{w}\chi^{-w}) \cdot \sigma_{(w,-w)}.\]
By definition, this is a homogeneous element in $$\Hom_{\CO_{X\!\times\!X}}\!(F_* K({X \times X}),K({X \times X}))_{(0,0)}.$$ We obtain
\begin{align*}
 \bar \sigma' &= \sum_{w} \bar{(f^p_{w} \otimes f^{-p}_{w}) (\chi^w \otimes \chi^{-w})\sigma_{(w,-w)}}\\
                   & =  \sum_{w} (f^p_{w} \otimes f^{-p}_{w}) \bar\sigma_{(w,-w)}.
\end{align*}

Now, we claim that  $\bar \sigma'$ gives the desired splitting on $Y \times Y$. To see that it is indeed a splitting, note, that 
$\sigma(1) = \sigma_{(0,0)}(1) = 1$. In particular, all other homogeneous components of $\sigma(1)$ vanish. Hence, multiplying one of these components with some element of the form $f_{w}^p\chi^{w} \otimes f_{w}^{-p}\chi^{-w}$ does not contribute to the degree-$(0,0)$ part of $\sigma'(1)$. We thus obtain $\sigma'(1) = \sigma(1) = 1$.  The same holds for $\bar \sigma'(1)$ which is just the restriction of $\sigma'(1)$ to the invariant functions.

It remains to show that $\bar\sigma'$ is compatible with the diagonal. We have
\[(1 - f^p_{w}\chi^{w} \otimes f^{-p}_{w}\chi^{-w}) \in I_X,\]
so $\sigma(g)$ and $\sigma'(g)$ differ only by an element of $I_X$. On the other hand, for some element $g \in F_*I_X$ we obtain $\sigma(g) \in I_X$, since $\sigma$ is compatible with the diagonal. Hence, $\sigma'(F_*I_X) \subset I_X$ holds. Since $\sigma'$ is of degree $(0,0)$ we also have $\sigma'(F_*I_{(0,0)}) \subset I_{(0,0)}$. Then we are done, since $I_{(0,0)}$ gives the ideal sheaf for the diagonal of $Y \times Y$ and $\bar{\sigma}'$ is just the restriction of $\sigma'$ to the degree $(0,0)$ part.
\end{proof}

Let us denote by $\supp_1 \sigma$ the set of degrees $w \in M$ such that 
the homogeneous component of degree $(w,-w)$ of $\sigma$ is non-trivial.

\begin{lemma}
\label{lem:sublattice-diagonal-split}
  Consider the subset $\bar M \subset M$ of those $u \in M$ such that $\D(u)$ is principal. Suppose there is a diagonal splitting $\sigma$ of $U$ satisfying
  \[\supp_1 \sigma \cap pM \subset p \bar M.\]
  Then there exists a diagonal splitting of $(Y,\Delta)$.
\end{lemma}
\begin{proof}
  This is just Lemma~\ref{lem:shifting} applied to the case $\div(f_w^p) = \D(w)$.
\end{proof}

Let's now consider the case that the torus action on $X$ is of complexity one. This means that $C=\Ysep$ is a curve. If $X$ is diagonally split, it is $F$-split as well and by Theorem~\ref{thm:compone} we know that $g(C,\Delta) \leq 1$, i.e. the curve is either elliptic or $\PP^1$.
\begin{prop}
  If $g(C,\Delta) > \nicefrac{1}{2}$ then an invariant diagonal splitting has to have a non-trivial component in a non-zero degree $(w,-w) \in pM \times pM$. 
\end{prop}
\begin{proof}
If an invariant splitting doesn't have a non-trivial component in a non-zero degree $(w,-w) \in pM \times pM$, then (by restricting to an open subset subset of $X$)  Lemma~\ref{lem:sublattice-diagonal-split} would provide us with a diagonal splitting of $(C,\Delta)$, which is impossible by Example~\ref{ex:diag-splitting-curves}.
\end{proof}

\begin{thm}
\label{thm:diag-nec-crit-cplx-one}
  Let $X$ be a complete diagonally split T-variety of complexity one. Then $C=\PP^1$ and we are in the cases $(1,*,*)$ or $(2,2,2)$ from Theorem~\ref{thm:compone}.
\end{thm}
\begin{proof}
  
  By restricting to an open subset $U$ of $X$ we assume that $Y=\Ysep=C$. Given a diagonal splitting $\sigma$ of form \eqref{eqn:diag} let $M' \subset M$ be the sublattice generated by $\supp_1 \sigma$. Note that by Proposition~\ref{prop:section-weight-criterion} the quotient $M'/(M' \cap pM)$  surjects to $M/pM \cong M \otimes_\ZZ \FF_p$ but this implies that $M' \cap pM  =pM'$. Indeed, given a $\ZZ$-basis $e_1',\ldots, e_\ell'$ of $M'$ we may consider its image in $M/pM$. By our condition on $\supp_1 \sigma$, the images of the basis vectors span $M/pM$. But this implies that they are linearly independent over $\FF_P$. Now, given an integral linear combination $u' = \sum_i \lambda_i e'_i$ of the basis elements which lies in $pM$ gives rise to a linear combinations 
$0 = \sum_i \overline{\lambda}_i \overline{e}_i'$ in $M/pM$. By linear independence the coefficients $\overline{\lambda}_i$ have to vanish. Hence, their representatives $\lambda_i$ are elements of $p\ZZ$ and $u'$ is an element of $pM'$.

  Let us now consider the case of pairs $(C,\Delta)$ of genus $1$.  Remember that  
\[\Hom_{\CO_{C}}(F_*\CO_{C}((p-1)\Delta),\CO_{C})^* \cong  H^0\left(C,\CO_C(\lceil(1-p)(K_C+\Delta)\rceil)\right).
\]
Hence, by Lemma~\ref{lemma:piotr} we have  $\deg \D(w) \leq 0$ and  $\deg \D(-w) \leq 0$ for every $w$ in the support of $\sigma$. By linearity this implies $\deg \D(w)=0$ for $w \in M'$. Again by Lemma~\ref{lemma:piotr}, $\D(w)$ has to be a principal divisor for $w$ in the support of $\sigma$ and hence for every $w \in M'$, as well.  Hence, we can take $\overline M = M'$ and apply  Lemma~\ref{lem:sublattice-diagonal-split}, using that $M'\cap pM=pM'$. We obtain a diagonal splitting of $(C, \Delta)$. But this is impossible by Example~\ref{ex:diag-splitting-curves}.

By Table \ref{table:triples}, the remaining cases we must rule out are those of
pairs $(C,\Delta)$ of genus larger than $ \nicefrac{3}{2}$, that is, the cases of the triples
$(2,2,r)$ ($r>2$), $(2,3,3)$, $(2,3,4)$, and $(2,3,5)$.
We set $\overline M$ to be the sublattice of $M'$ consisting of those $u$ such that $\deg \D(u) = 0$. Our first claim is that $\supp_1 \sigma \cap pM \subset \overline M$.  Indeed, if $\deg \D(w) > 0$ for some $w \in pM$,  one can check case by case that $\deg \D(w)$ would be at least $p \cdot (2 - \deg \Delta)$. Now, we would have  
\[(1-p)\deg(\Delta + K_C) - \deg\D(w)  < 0\]
 and there cannot be a non-trivial homomorphism in degree $w$. On the other hand, if $\deg \D(w) < 0$ than we have $\deg \D(-w) > 0$. Hence, we must have  $\deg\D(w)=0$ for all degrees in $\supp_1 \sigma \cap pM$.

We can apply the same methods as in the genus $1$ case if $\D(w)$ is integral for every $w \in \overline M$. If we are in the case $(2,3,5)$ this has to hold true, since there is no way to obtain $\nicefrac{a}{2} + \nicefrac{b}{3} + \nicefrac{c}{5}$ being an integer without all the summands being integers.

For the remaining cases, we will use the  diagonal splitting of $U$ to construct  a diagonal splitting of $(C;\Delta_+,\Delta-)$. Here, writing $\Delta=a_1[c_1] + a_2[c_2]+a_3[c_3]$, we take $\Delta_+=a_1[c_1] + a_3[c_3]$ and $\Delta_-=a_2[c_2] + a_3[c_3]$. Note that by properly ordering $a_1,a_2,a_3$, we have $\deg \Delta_+ = a_1+a_3 > 1$ and $\deg \Delta_-=a_2+a_3 >1 $ so as before, by Example~\ref{ex:diag-splitting-curves} we will obtain a contradiction.

We will discuss the case $\Delta = \nicefrac{2}{3}[c_1] + \nicefrac{2}{3}[c_2]+\nicefrac{1}{2}[c_3]$ in detail; the other cases follow similarly. 
We wish, for any $w\in\supp_1\sigma \cap pM$, to produce a function $f_w$ as in Lemma \ref{lem:shifting}. Write such $w$ as $w = (\ell p)w'$ with $\ell \in \NN$ and $w' \in M'$ a primitive lattice element. Now, we have seen above that $\D(w')$ is of degree $0$. If $\D(\ell w')$ is integral, then we set $D_w = \D(\ell w')$ and have $\D(w) = pD_w$. Since $D_w$ has degree zero, it is principal, that is, $D_w=\div f_w$ for some rational function $f_w$. Furthermore, this $f_w$ satisfies the requirements of Lemma \ref{lem:shifting}, since
$$
\lceil (p-1)\Delta + \D(\pm w) \mp  pD_w\rceil = \lceil (p-1)\Delta \rceil \geq  \lceil (p-1)\Delta_+\rceil,\lceil (p-1)\Delta_-\rceil.
$$

Assume now instead that $\D(\ell w')$ is not integral. Since it has degree $0$,  up to changing the roles of $c_1$ and $c_2$  we have $\D(w')=\nicefrac{1}{3}[c_1] - \nicefrac{1}{3}[c_2] + D_0$, with $D_0$ some integral divisor of degree $0$. This means that 

\begin{align*}
  \lceil (p-1)\Delta + \D(w) \rceil &= \left\lceil\frac{(2+\ell) p - 2}{3}\right\rceil[c_1]+\left\lceil\frac{(2-\ell) p - 2}{3}\right\rceil[c_2] + \left\lceil \frac{p}{2}\right\rceil[c_3] + \ell p D_0.
\end{align*}

Now, if $\ell \equiv 0 \mod 3$ , then $\D(\ell w')$ is integral so the case above applies. Suppose instead that  
$\ell \equiv  2 \mod 3$. Then we obtain
\begin{align*}
   \lceil (p-1)\Delta + \D( w) \rceil &= \left\lceil\frac{(2+2) p - 2}{3}\right\rceil[c_1]+\left\lceil\frac{(2-2) p - 2}{3}\right\rceil[c_2] + \left\lceil\frac{p}{2}\right\rceil[c_3] +p D_w'\\
   &= \left\lceil\frac{4p - 2}{3}\right\rceil[c_1]+ \left\lceil\frac{p}{2}\right\rceil[c_3]+ p D_w';\\
   \lceil (p-1)\Delta + \D(-w) \rceil &= \left\lceil\frac{(2-2) p - 2}{3}\right\rceil[c_1]+\left\lceil\frac{(2+2) p - 2}{3}\right\rceil[c_2] + \left\lceil\frac{p}{2}\right\rceil[c_3] -p D_w'\\
   &= \left\lceil\frac{4p - 2}{3}\right\rceil[c_2]+ \left\lceil\frac{p}{2}\right\rceil[c_3]+ p D_w'
\end{align*}
with $D'_w$ being an integral divisor of degree $0$, hence of the form $\div f_w$ for some rational function $f_w$.
Since
\begin{align*}
\lceil (p-1)\Delta + \D(w) - pD_w'\rceil &\geq  \lceil (p-1)\Delta_{+}\rceil\\
\lceil (p-1)\Delta + \D(-w) + pD_w'\rceil &\geq  \lceil (p-1)\Delta_{-}\rceil
\end{align*}
the function $f_w$ fulfills the requirements for Lemma \ref{lem:shifting}.

If instead $\ell \equiv -2 \mod 3$, a similar analysis also produces a function $f_w$ satisfying the requirements of Lemma \ref{lem:shifting}. 
Now, applying Lemma~\ref{lem:shifting} we obtain a diagonal splitting of $(C;\Delta_+,\Delta_-)$. But as we have seen, this is impossible.
\end{proof}

\begin{ex}[Blowup of a flag variety (continued)]\label{ex:bl2}
Once more consider the variety $\widetilde W$ from Example~\ref{ex:main}. Remember, that the piecewise linear function from Lemma \ref{lemma:homs-separated} describing the homogeneous components of $\Hom(F_*\CO_{\widetilde W}, \CO_{\widetilde W})$ was given in Example~\ref{ex:bl1} by
\[ h(a,b) =  \max\{-a,0\}[1] \;+\; \max\{-b,0\}[1] \;+\; \max\{a+b,0\}[\infty].\]

Now, for every pair of integers $w=(a,-b)$ with $0\leq a,b \leq p-1$ we set $w'=(a-p,-b)$ and we have $w,w' \in (p-1)P_{\widetilde W}$.
Moreover, we obtain\\
\begin{minipage}[t]{0.5\textwidth}
\begin{align*}
   h(w) &=  b[1] \;+\; \max\{a-b,0\}[\infty], \\
   h(w') &=  (p-a)[0]\;+\; b[1],
\end{align*}
\end{minipage}
\begin{minipage}[t]{0.5\textwidth}
\begin{align*}
  h(-w) &= a[0] \;+\; \max\{b-a,0\}[\infty],\\
  h(-w')&= (b+p-a)[\infty],
\end{align*}
\end{minipage}
\smallskip

Recall (Remark \ref{rem:duality}) that there is an isomorphism 
\[
\Hom_{\CO_{\PP^1}}(F_*\CO_{\PP^1}(D), \CO_{\PP^1}) \cong H^0(\PP^1, \CO((1-p)K_{\PP^1}-D)).
\]
We will denote this correspondence by the symbol $\triangleq$. 

For $w$ and $w'$ as above and $K = K_{\PP^1} = -[0]-[\infty]$ we consider 
\begin{align*}
  \sigma_{(w,-w)} &\triangleq \sum_{i=0}^{p-a-1} y^{1-p+i}(y-1)^{p-1}\otimes y^{-i}\quad \\
  &\in  H^0(\CO(-K-\lceil h(w)\rceil)) \otimes H^0(\CO(-K-\lceil h(-w)\rceil)), \\
  \sigma_{(w',-w')} &\triangleq \sum_{i=p-a}^{p-1} y^{1-p+i}(y-1)^{p-1} \otimes y^{-i}\\
 &\in H^0(\CO(-K-\lceil h(w')\rceil)) \otimes H^0(\CO(-K-\lceil h(-w)'\rceil))
\end{align*}
as elements of 
\[\Hom(F_*\CO_{\widetilde W}, \CO_{\widetilde W})_{w} \otimes \Hom(F_*\CO_{\widetilde W}, \CO_{\widetilde W})_{-w} =\Hom(F_*\CO_{\widetilde W \times \widetilde W}, \CO_{\widetilde W \times \widetilde W})_{(w,-w)}\]
and $\Hom(F_*\CO_{\widetilde W \times \widetilde W}, \CO_{\widetilde W \times \widetilde W})_{(w',-w')}$, respectively.

We set $\sigma$ to be the sum of all these $\sigma_{(w,-w)}$ and $\sigma_{(w',-w')}$. Then we obtain
\begin{align}
 \bar \sigma_{[w]} &\triangleq  \sum_{i=0}^{p-1} y^{1-p+i}(y-1)^{p-1}\otimes y^{-i} \label{eq:4}\\ \nonumber
                       &= \frac{(y\otimes 1 - 1\otimes y)^{p-1}}{y^{p-1} \otimes 1} \cdot (y-1)^{p-1} \;\in H^0(\PP^1 \times \PP^1,\CO(-K_{\PP^1 \times \PP^1}-\text{diag}))
\end{align} 
Hence, we have $\bar \sigma_{[w]} = \bar \sigma_{[0]}$ for every $w \in M$. Moreover,  $\bar \sigma_{[w]}$ is compatible with the diagonal. It remains to show that $\sigma$ is actually a splitting. To see this, note, that $\sigma_0$ is the only non-trivial homogeneous component $\sigma_{(w,-w)}$ with $w \in pM$. Moreover, $\bar{\sigma}_0 = \bar{\sigma}_{[0]}$ defines a splitting for $\PP^1 \times \PP^1$, since the monomial $1$ occurs with coefficient $1$ in (\ref{eq:4}). Hence, we have $\sigma(1) = \sigma_0(1) = 1$.

We just proved that the blowup $\widetilde W$ of the flag variety $W$ is diagonally split. This implies also that the blow up in only one of the curves and $W$ itself are diagonally split. The latter was previously known, since all flag varieties are diagonally split by \cite{ramanathan:87a}.
\end{ex}

\begin{ex}
 Consider the blowup $X$ of $\PP^1 \times \PP^1 \times \PP^1$ in a curve of degree $(0,1,1)$; this is number 4.8 in the classification of Fano threefolds by Mori and Mukai \cite{mori:81a}. There is a $\GG_m^2$-action here defined by the weight matrix
\[
\begin{array}{rrrrrrrl}
  &u_0&u_1&v_0&v_1&w_0&w_1&
\vspace{2mm}\\
 \ldelim({2}{0.5ex}
  &1&0&0&0&0&0&\rdelim){2}{0.5ex} \\
  &0& 0&1&0&-1&0& \quad
\end{array}
\]
where the $u_i,v_i,w_i$ are homogeneous coordinates on the three factors of $\PP^1$.
We may assume that the center of the blow up is  the curve $C=\{1\} \times V(v_0w_0-v_1w_1)$. The quotient is again a non-separated $\PP^1$ with the points $0,1,\infty$ doubled. The separation is just $\PP^1$ and the corresponding quotient map is given by
\[(u_0:u_1, v_0:v_1,w_0:w_1) \mapsto (v_0w_0:v_1w_1)\]
and we have two prime divisors in $X \setminus X^\circ$ with corresponding one-parameter subgroups $\pm \rho \in N \cong \ZZ^2$, with $\rho=(0,1)$. This and the piecewise linear function $h:P_X \to \Div_\QQ \PP^1$ can be obtained similarly to Example~\ref{ex:main} or read off from the data  given in \cite{suess:14a}. For $h$ we obtain
\[ h(a,b) =  \max\{-a,0\}[1] \;+\; \max\{-b,0\}[1] \;+\; \max\{a,0\}[\infty].\]

One checks that $\sigma_{(w,-w)}$ as defined in Example \ref{ex:bl2} is again an element of  $H^0(\CO(-K-\lceil h(w)\rceil)) \otimes H^0(\CO(-K-\lceil h(-w)\rceil))$ and similarly for  $\sigma_{(w',-w')}$. Hence, we can again take the sum of all $\sigma_{(w,-w)}$ and  $\sigma_{(w',-w')}$ to obtain a diagonal splitting for $X$.
\end{ex}

\bibliographystyle{alpha}
\bibliography{t-frobenius}
\end{document}